\documentclass[12pt]{article}
\usepackage{hyperref,amsmath,amsthm,amsfonts,amssymb,amscd,enumerate,tikz}

\usetikzlibrary{positioning}

\theoremstyle{plain}
\newtheorem{theorem}{Theorem}[section]
\newtheorem{proposition}[theorem]{Proposition}

\newtheorem{corollary}[theorem]{Corollary}
\newtheorem{lemma}[theorem]{Lemma}

\DeclareMathOperator{\piprod}{\raisebox{-0.1em}{\huge{$\pi$}}\kern -0.2em}

\theoremstyle{definition}

\newtheorem{definition}[theorem]{Definition}
\newtheorem{example}[theorem]{Example}
\newtheorem{examples}[theorem]{Examples}

\newtheorem{nonexample}[theorem]{Nonexample}
\newtheorem{remark}[theorem]{Remark}

\newtheorem*{example*}{Example}

\newcommand{\ca}{\mathcal {A}}

\newcommand{\cac}{\mathcal {C}}

\newcommand{\cE}{\mathcal {E}}
\newcommand{\cf}{\mathcal {F}}
\newcommand{\cg}{\mathcal {G}}
\newcommand{\ch}{\mathcal {H}}

\newcommand{\cm}{\mathcal {M}}

\newcommand{\cp}{\mathcal {P}}

\newcommand{\cs}{\mathcal {S}}

\newcommand{\cu}{\mathcal {U}}

\newcommand{\cw}{\mathcal {W}}

\newcommand{\cz}{\mathcal {Z}}
\newcommand{\cc}{{\mathbb C}}

\newcommand{\qq}{{\mathbb Q}}
\newcommand{\rr}{{\mathbb R}}
\newcommand{\sphere}{{S}}

\newcommand{\zz}{{\mathbb Z}}

\newcommand{\minus}{{-1}}

\newcommand{\ii}{|I|}

\newcommand{\bA}{{\mathbf A}}
\newcommand{\bB}{{\mathbf B}}

\newcommand{\bR}{{\mathbb R}}

\newcommand{\bx}{{\mathbf x}}

\newcommand{\boldC}{{\mathbf C}}
\newcommand{\bC}{{\mathbb C}}

\newcommand{\bE}{{\mathbf E}}
\newcommand{\bG}{{\mathbf G}}

\newcommand{\boldb}{{\mathbf b}}

\newcommand{\zl}{{\cz_L}}
\newcommand{\zlone}{{\cz_{L^1}}}
\newcommand{\zll}{{\cz_{L'}}}
\newcommand{\zx}{{\zl(\bA,\bB)}}

\newcommand{\tw}{{\tilde{w}}}
\newcommand{\tx}{{\tilde{x}}}

\newcommand{\wt}{\widetilde}
\newcommand{\wa}{\widetilde {A}}
\newcommand{\wb}{\widetilde {B}}
\newcommand{\bwa}{\widetilde {\mathbf A}}
\newcommand{\bwb}{\widetilde {\mathbf B}}
\newcommand{\ab}{(\bA,\bB)}
\newcommand{\wab}{(\bwa,\bwb)}

\newcommand{\ww}{\widetilde {W}}
\newcommand{\wx}{\widetilde {X}}
\newcommand{\wy}{\widetilde {Y}}

\newcommand{\wzl}{\widetilde {\cz}_L}
\newcommand{\wzlone}{\widetilde {\cz}_{L^1}}
\newcommand{\wzll}{\widetilde {\cz}_{L'}}

\newcommand{\wcac}{\widetilde {\cac}}
\newcommand{\wcm}{\widetilde {\cm}}
\newcommand{\wcu}{\widetilde {\cu}}

\newcommand{\ol}{\overline}

\newcommand{\oh}{\ol{H}}

\newcommand{\jcheck}{J\,\check{}}
\newcommand{\fjcheck}{f(J)\,\check{}}
\newcommand{\mfj}{\overset{.}{M}_{\fjcheck}}
\newcommand{\primecheck}{f(J')\,\check{}}
\def\clap#1{\hbox to 0pt{\hss#1\hss}}

\newcommand{\comment}[1]{}

\newcommand{\gd}{\delta}

\newcommand{\gf}{\varphi}
\newcommand{\gi}{\iota}

\newcommand{\gs}{\sigma}

\newcommand{\gG}{\Gamma}
\newcommand{\gD}{\Delta}

\newcommand{\gL}{\Lambda}

\newcommand{\od}{{\overline{D}}}

\newcommand{\vertex}{\operatorname{Vert}}
\newcommand{\edge}{\operatorname{Edge}}

\newcommand{\flag}{\operatorname{Flag}}

\newcommand{\Lk}{\operatorname{Lk}}

\newcommand{\cat}{\operatorname{CAT}}
\newcommand{\cone}{\operatorname{Cone}}

\newcommand{\supp}{\operatorname{Supp}}

\newcommand{\raag}{\operatorname{RAAG}}
\newcommand{\rab}{\operatorname{RAB}}
\newcommand{\racg}{\operatorname{RACG}}
\newcommand{\racs}{\operatorname{RACS}}

\newcommand{\Econe}{(\boldsymbol{\cone}(\mathbf E), \mathbf E)}

\newcommand{\bman}{(\mathbf M, \boldsymbol{\partial}\mathbf M)}

	{
\end{enumerate}
}
\newenvironment{enumerate1*}{
\begin{enumerate}[\upshape (*1)]}%
	{
\end{enumerate}
}
\newenvironment{enumeratei}{
\begin{enumerate}[\upshape (i)]}%
	{
\end{enumerate}
}
\newenvironment{enumeratei'}{
\begin{enumerate}[\upshape (i)$'$]}%
	{
\end{enumerate}
}

\newenvironment{enumeratea}{
\begin{enumerate}[\upshape (a)]}{
\end{enumerate}
}
\newenvironment{enumeratea'}{
\begin{enumerate}[\upshape (a)$'$]}{
\end{enumerate}
}

 \numberwithin{equation}{section} 

\begin{document}
\title{Right-angularity, flag complexes, asphericity}
\author{Michael W. Davis\thanks{The  author was partially supported by NSF grant DMS-1007068 and by the Institute for Advanced Study.}  }
\date{}
\maketitle
\begin{abstract}
The ``polyhedral product functor'' produces a space from a simplicial complex $L$ and a collection of pairs of spaces, $\{(A(i),B(i))\}$, where $i$ ranges over the vertex set of $L$.  We give necessary and sufficient conditions for the resulting space  to be aspherical. There are two similar constructions, each of which starts with a space $X$ and a collection of subspaces, $\{X_i\}$, where $i \in \{0,1\dots, n\}$, and then produces a new space.  We also give conditions for the results of these constructions to be aspherical.  All three techniques can be used to produce examples of closed aspherical manifolds.
	 \smallskip

	\paragraph{AMS classification numbers.} Primary: 20F65, 57M07, 57M10. \\
	Secondary: 20F36, 20E42, 20F55. \smallskip

\end{abstract}

\section*{Introduction}
This paper concerns three similar methods for  combining certain spaces (given as the  input data)  to construct  new spaces.  These constructions are closely related to a standard construction of a cubical complex with the link at each vertex a specified simplicial complex (cf.\ \cite[p.\ 212]{bh}, \cite[\S1.2]{dbook} or Example~\ref{ex:com}, below).  For each construction we prove a theorem which gives conditions for the resulting space to be aspherical.  We also give conditions for the result to be a closed manifold.  

The first result, Theorem~\ref{t:asphericalflag}, gives necessary and sufficient conditions for the polyhedral product of a collection of pairs of spaces, $\{(A(i),B(i))\}_{i\in I}$, to be aspherical.  (The notion of a ``polyhedral product'' is defined in \S\ref{ss:defex} as well as in \cite{bbcg,bbcg2}.) The second result, Theorem~\ref{t:flag}, gives necessary and sufficient conditions for the result of applying a reflection group trick to a corner of spaces to be aspherical.  (The notion of a ``corner of spaces'' is defined in \S\ref{s:corners}: roughly, it is a space $X$ together with a collection of subspaces $X_i$, indexed by $I(n)=\{0,1\dots, n\}$, such that all possible intersections of the $X_i$ are nonempty.)  The third construction also starts with a corner of spaces $X$, but now as additional data we are given a simplicial complex $L$ and a coloring of its vertices $f:L\to\gD$ (where $\gD$ is the simplex on $I(n)$). The third result, Theorem~\ref{t:asphericalcorner}, gives sufficient conditions for the ``pullback,'' $f^*(X)$, to be aspherical. A common thread in all three results  is that certain auxiliary simplicial complexes need to be flag complexes. 

Next, we make a few remarks about each construction in turn.  Given a simplicial complex $L$ and a collection of pairs $\ab=\{(A(i),B(i)\}_{i\in I}$, indexed by the vertex set $I$ of $L$, the polyhedral product $\zl\ab$ is defined as a certain subspace of the product $\prod_{i\in I} A(i)$.  If each $A(i)$ is simply connected and each $B(i)$ is connected, then $\zl\ab$ is simply connected; however, if either of these conditions fail, then the fundamental group of $\zl\ab$ is usually nontrivial.  We compute this group in Theorem~\ref{t:pi1}. Roughly speaking, it is a ``generalized graph product'' of the fundamental groups of the $A(i)$.  The conditions for $\zl\ab$ to be aspherical are the following:
\begin{enumeratei}
\item
each $A(i)$ is aspherical,
\item
for each ``nonconelike vertex'' $i\in I$, each path component of $B(i)$ is aspherical and its fundamental group injects into $\pi_1(A(i))$, and 
\item
$L$ is a flag complex.
\end{enumeratei}

Suppose $X$ has a corner structure $\cm=\{X_i\}_{i\in I(n)}$.  (The $X_i$ are called ``mirrors.'')  Suppose further that  $(W,S)$ is a spherical Coxeter system with its set of fundamental involutions $S$ indexed by $I(n)$.  (For example, $W$ could be $(\boldC_2)^{I(n)}$ where $\boldC_2$ denotes the group of order two.)  The ``basic construction'' of \S\ref{ss:basicconst} yields a space with $W$-action, $\cu(W,X)$. Necessary and sufficient conditions for  $\cu(W,X)$ to be aspherical are the following:
\begin{enumeratei'}
\item
$X$ is aspherical,
\item
for each $i\in I(n)$, any path component of the inverse image of $X_i$ in the universal cover $\wx$ of $X$ is acyclic and  similarly for the inverse images of arbitrary intersections of the $X_i$, and
\item
there is an ``induced mirror structure'' $\wcm$ on $\wx$ and an associated simplicial complex $N(\wcm)$, called its ``nerve,''  the  condition  being that $N(\wcm)$ is a flag complex.
\end{enumeratei'}

For the third construction, we are given a corner of spaces $X$; however, in place of a Coxeter system, suppose we have a simplicial complex $L$ with a coloring of its vertices,  $f:L\to \gD$.  The space $f^*(X)$ is formed by pasting together copies of $X$ in a fashion specified by $L$.   Sufficient conditions for $f^*(X)$ to be aspherical are then conditions (i)$'$, (ii)$'$, (iii)$'$, above, as well as, the condition that $L$ be a flag complex.

A corner of spaces is a \emph{corner of manifolds} if the space in question is a manifold with corners and if the mirrors are strata of codimension one.  There are two good sources of corners of manifolds which satisfy conditions (i)$'$, (ii)$'$, (iii)$'$ above:
\begin{itemize}
\item
(\emph{Products of aspherical manifolds with boundary}). For each $i\in I(n)$, suppose $(M(i),\partial M(i))$ is an aspherical manifold with aspherical, $\pi_1$-injective boundary.  Put $M=\prod_{i\in I(n)} M(i)$ and let $M_i$ be the set of points $x\in M$ such that $x_i\in \partial M(i)$.  Then $\{M_i\}_{i\in I(n)}$ is a corner structure on $M$ satisfying the three conditions (cf.\ \S\ref{ss:aspflag}).
\item
(\emph{Borel-Serre compactifications}).  Suppose $\bG$ is an algebraic group defined over $\qq$ with real points $G$.  Let $D$ be the associated symmetric spaces.  Suppose $\gG$ is a  torsion-free, arithmetic subgroup of $G$ so that $\gG\backslash D$ is a manifold.  If the $\qq$-rank  of $\bG$ is greater than $0$, then $\gG\backslash D$ is not compact.  If the $\qq$ rank of $\bG$ is $n+1$, then $\gG\backslash D$ has a ``Borel-Serre compactification'' $M$ which is a corner of manifolds over $I(n)$.  As we shall see in \S\ref{ss:aspflag}, the corner structure on $M$ satisfies  conditions (i)$'$, (ii)$'$, (iii)$'$, above.
\end{itemize}

I became interested in these examples after  conversations with Tam Nguyen Phan, a graduate student at the University of Chicago.  She 
showed that one could obtain aspherical manifolds by applying the reflection group trick to  Borel-Serre compactications.  This led me to start thinking about  more general situations and conditions (i)$'$, (ii)$'$, (iii)$'$.

My thanks go to  the referee for  carefully reading the manuscript.

\section{Basic definitions}\label{s:basic}
\subsection{Mirror structures} 
A \emph{mirror structure} on a space $X$ over a set $I$ is a family of closed subspaces $\cm=\{X_i\}_{i\in I}$.  
For each subset $J\le I$, put 
	\begin{equation}\label{e:XJ}
	X_J:=\bigcap_{i\in J} X_i, \qquad X^J:=\bigcup_{i\in J} X_i
	\end{equation}
Also, $X_\emptyset =X$ and $X^\emptyset =\emptyset$.  
For each $x\in X$, put  
	\begin{equation}\label{e:Ix}
	I(x):=\{i\in I\mid x\in X_i\}.
	\end{equation}  
The \emph{nerve of the mirror structure} is the simplicial complex $N(\cm)$ with vertex set $I$ defined by the requirement that a subset $J\le I$ spans a simplex if and only if $X_J\neq \emptyset$.

Suppose that $X$ is a mirrored space over $I$, that $G=\pi_1(X)$ and that $p:\wx\to X$ is the universal covering.  One can then define the ``induced mirror structure on $\wx$,'' as follows.  For each $i\in I$, let $E_i$ be the $G$-set of path components of $p^\minus(X_i)$.  Let $E$ be the disjoint union of the $E_i$. 
The \emph{induced mirror structure} $\wcm=\{\wx_e\}_{e\in E}$ on $\wx$ is defined by $\wx_e:= e$.

\subsection{Coxeter groups}\label{ss:cox}  
A \emph{Coxeter matrix} $(m(i,j))$ on a set $I$ is an $I\times I$ symmetric matrix with $1$'s on the diagonal and with off-diagonal entries integers $\ge 2$ or the symbol $\infty$.  A Coxeter matrix defines a \emph{Coxeter group} $W$ with generating set $S:=\{s_i\}_{i\in I}$ and with relations:
	\[
	(s_is_j)^{m(i,j)}=1, \quad \text{for all $(i,j)\in I\times I$}
	\]
Since each diagonal entry is equal to $1$, this entails that each $s_i$ is an involution.  (If $m(i,j)=\infty$, the above relation is simply omitted from the presentation.)  The pair $(W,S)$ is a \emph{Coxeter system}.  For  a subset  $J\le I$, let $S_J:=\{s_i\mid i\in J\}$ and let $W_J:=\langle S_J\rangle$ be the subgroup generated by $S_J$.  $W_J$ is called a \emph{special subgroup} of $W$.  The pair $(W_J,S_J)$ is a Coxeter system (cf.~\cite[Thm.~4.1.6, p.~45]{dbook}).  If $W_J$ is a finite group, then $J$ is a \emph{spherical subset} of $I$ and any conjugate of $W_J$ is a \emph{spherical subgroup}.  Let $\cs(W,S)$ denote the poset of spherical subsets of $I$.  There is an associated simplicial complex $L(W,S)$ called the \emph{nerve of $(W,S)$}, whose poset of simplices is $\cs(W,S)$.  In other words, the vertex set of $L(W,S)$ is $I$ and a nonempty subset of $I$ spans a simplex if and only if it is spherical.	

A Coxeter matrix is \emph{right-angled} if all its off-diagonal are equal to $2$ or $\infty$.  Similarly, the associated Coxeter system is a \emph{right-angled Coxeter system} (a $\racs$) and the group is a \emph{right-angled Coxeter group} (a $\racg$).

\subsection{Buildings}\label{ss:bldgs} 
A \emph{chamber system} is set $\cac$ (of ``chambers'') together with a family of equivalence relations (``adjacency relations'') indexed by $I$.  For $i\in I$, an $i$-equivalence class is called an \emph{$i$-panel}.  Two chambers are \emph{$i$-adjacent} if they are $i$-equivalent and not equal.  A \emph{gallery} in $\cac$ is a sequence $C_0, \dots, C_k$ of adjacent chambers.  If $C_{m-1}$ is $i_m$-adjacent to $C_m$, then the \emph{type} of the gallery is the word $i_1\cdots i_k$ in $I$. 

Suppose $(W,S)$ is a Coxeter system, where $S=\{s_i\}_{i\in I}$ is indexed by $I$.  Given a word $i_1\cdots i_k$ in $I$, its \emph{value} in $W$ is the element $s_{i_1}\cdots s_{i_m}$ of $W$.  A \emph{building} of type $(W,S)$ is a chamber system $\cac$ of type $I$ together with a \emph{Weyl distance} $\gd:\cac\times \cac \to W$ so that 
\begin{itemize}
\item
each panel has at least two elements, and
\item
given chambers $C,D \in \cac$ with $\gd(C,D)=w$, there is a gallery of minimal length from $C$ to $D$ of type $i_1\cdots i_n$ if and only if  $s_{i_1}\cdots s_{i_n}$ is a reduced expression for $w$.
\end{itemize}
Given a subset $J\le I$, chambers $C$ and $D$ are said to belong to the same \emph{$J$-residue} if they can be connected by a gallery such that all letters in its type belong to $J$.  A singleton $\{C\}$ is a $\emptyset$-residue; an $i$-panel is the same thing as an $\{i\}$-residue (cf.\ \cite{ab}, \cite{ronan}).

A building is \emph{spherical} if its type is a spherical Coxeter system.  A building  is a \emph{right-angled building} (a $\rab$) if its type is a $\racs$.
\begin{example}\label{ex:rank1.1}(\emph{Products of rank 1 buildings}).
A rank one building (i.e., a building of type $\bA_1$) is just a set in which every two distinct elements are adjacent.  Suppose $\cac_0,\dots, \cac_n$ are rank one buildings.  The product $\cac_0 \times \cdots \times \cac_n$ can be given the structure of a building of type $\bA_1\times \cdots \times\bA_1$ as follows.  For $i\in \{0,1,\dots ,n\}$, elements $(c_0,\dots ,c_n)$ and $(d_0,\dots , d_n)$ are defined to be $i$-equivalent if $c_j=d_j$ for all $j\neq i$. This gives the product the structure of a chamber system; the Weyl distance is then defined in the obvious manner.  The associated Coxeter group is $\boldC_2\times \cdots \times \boldC_2$; hence, the product is spherical and right-angled. (More generally one can define the product of chamber systems or of buildings in a similar fashion, cf.\  \cite[p.~2]{ronan}.)
\end{example}

\subsection{Flag complexes}\label{ss:flag} 
Suppose $L$ is a simplicial complex with vertex set $I$.  Let $\cs(L)$ be the poset of (vertex sets of) simplices in $L$, including the empty simplex.   
\begin{definition}\label{d:flag}
$L$ is a \emph{flag complex}  if for any subset $J$ of $I$ such that every two elements of $J$ bound an edge of $L$, then $J\in \cs(L)$.  (In other words, $L$ is flag if every nonsimplex contains a nonedge.)
\end{definition}

Any simplicial graph $L^1$ determines a flag complex $L$ with $L^1$ as its $1$-skeleton as follows: a finite subset $J\subset \vertex(L^1)$ spans a simplex of $L$ if and only if it spans a complete subgraph in $L^1$.  (Combinatorialists say that $L$ is the \emph{clique complex of $L^1$}.)

\begin{example}\label{ex:posetflag}
If $\cp$ is a poset, then the simplicial complex $\flag (\cp)$ whose simplices are all nonempty finite chains $p_0<\cdots < p_n$ is a flag complex.  This means that the barycentric subdivision of any convex cell complex is a flag complex (cf.\ \cite[Example~A.5, p.~663]{ab} or  \cite[Appendix~A.3]{dbook}).
\end{example}

More details about flag complexes  can be found in \cite[Appendix~A.1.2]{ab} or \cite[Appendix~A.1]{dbook}.

\paragraph{The complex $K(L)$ and its mirror structure.} 
The space $K(L)$ is defined to be the simplicial complex $\flag(\cs(L))$ (cf.\  Examples~\ref{ex:posetflag} (3)).  The complex $K(L)$ is isomorphic to the cone on the barycentric subdivision of $L$ (the empty set provides the cone point).  There is a canonical mirror structure over $I$ on $K(L)$ defined by $K(L)_i:= \flag(\cs(L)_{\ge \{i\}})$.  In other words, $K(L)_i$ is the star of the vertex $i$ in the barycentric subdivision of $L$.  It  follows from the definition in \eqref{e:XJ} that $K(L)_J$ is nonempty if and only if $J\in \cs(L)$, in which case,
	\begin{equation}\label{e:mirror}
	K(L)_J=\flag(\cs(L)_{\ge J}).
	\end{equation}
	
\begin{definition}\label{d:davischam}
If $L=L(W,S)$ is the nerve of a Coxeter system, then $K(W,S):=K(L(W,S))$ is its \emph{Davis chamber}.  (In other words, $K(W,S)=\flag(\cs(W,S))$.)
\end{definition}

\subsection{The basic construction}\label{ss:basicconst}  
Suppose $X$ is a mirrored space with its mirrors indexed by a set $I$ and that $(W,S)$ is a Coxeter system with its fundamental set of generators $S$ indexed by the same set $I$. Define an equivalence relation $\sim$ on $W\times X$ by 
	\[
	(w,x)\sim (w',x') \iff x=x' \text{ and } wW_{I(x)}=w'W_{I(x)},
	\]
where $I(x)$ is defined by \eqref{e:Ix}.  The \emph{basic construction} on this data is the $W$-space, 
	\begin{equation}\label{e:basic}
	\cu(W,X):=(W\times X)/\sim.
	\end{equation}
More generally, if $\cac$ is a building of type $(W,S)$, define $\sim$ on $\cac \times X$ by
	\[
	(C,x)\sim (C',x') \iff x=x' \text{ and $C$, $C'$ belong to the same $I(x)$-residue.}
	\]
As before, $\cu(\cac, X):= (\cac \times X)/\sim$.

The simplex $\gD$ on $I$ has a mirror structure $\{\gD_i\}_{i\in I}$ where $\gD_i$ is defined to be the codimension-one face opposite the vertex $i$.  The \emph{classical realization} of a building $\cac$ is defined to be $\cu(\cac, \gD)$.  When $\cac$ is the thin building $W$, then the classical realization $\cu(W,\gD)$ is called the \emph{Coxeter complex}.  If $W$ is a spherical, we can identify $\gD$ with a spherical simplex and $\cu(W,\gD)$ with the round sphere.  If $\cac$ is a spherical building of type $(W,S)$, then $\cu(\cac,\gD)$ is called the \emph{spherical realization} of $\cac$.  For example, the spherical realization of a product of rank one buildings $\cac_0\times \cdots\times\cac_n$ (cf.\  Example~\ref{ex:rank1.1}) is equal to their join, 
\[
\cu(\cac_0\times\cdots \times \cac_n, \gD^n)=\cac_0*\cdots *\cac_n).
\]

\begin{definition}\label{d:standard}
The \emph{standard realization} of a building $\cac$ of type $(W,S)$ is $\cu(\cac,K)$, where $K=K(W,S)$, the Davis chamber of $(W,S)$ (cf.\  Definition~\ref{d:davischam}).
\end{definition}

\begin{examples}\label{ex:flags} (\emph{More examples of flag complexes}). 

(1) If $E_0, \dots E_n$ are discrete sets, then their join, $E_0*\cdots *E_n$ is an $n$-dimensional simplicial complex and a flag complex.
(When each $E_i$ has at least two elements this is the classical realization of $E_0\times \cdots \times E_n$, a spherical building of type $\bA_1\times \cdots \times \bA_1$; so, the fact that it is a flag complex is a special case of (2), below.)

(2) The classical realization of any building is a flag complex (cf.\  \cite[Ex.~4.50, p.{} 190]{ab}).  
In particular, the Coxeter complex of any Coxeter system $(W,S)$ is a flag complex.

(3) Let $\ch$ be any simplicial arrangement of linear hyperplanes in $\rr^n$.  Then the nonzero faces of $\ch$  define a simplicial complex $\gL(\ch)$ which is a triangulation of $\sphere^{n-1}$.  It turns out that $\gL(\ch)$ is a flag complex (see \cite[p. 29]{oldbrown}).

(4) The Deligne complex (cf.\ \cite{deligne} or \cite{cdjams}) of an Artin group of spherical type is a flag complex.  (This is proved  in \cite[Lemma~4.3.2]{cdjams}, where the proof occupies pages 623--626.)  Similarly, the Deligne complex of any real simplicial arrangement is a flag complex.

(5) The link of a simplex in a flag complex is easily seen to be a flag complex (cf.\ \cite[Remarks~5.16, p.\,210]{bh}).
\end{examples}

\section{Polyhedral products}\label{s:pp}

\subsection{Definitions and examples}\label{ss:defex}
Suppose we are given a family of pairs of spaces $(\bA,\bB)=\{(A(i),B(i))\}_{i\in I}$.  We shall always assume that  $A(i)$ is path connected and that $B(i)$ is a nonempty, closed subspace of $A(i)$.  Let $\bx:=(x_i)_{i\in I}$ denote a point in the product $\prod _{i\in I} A(i)$.  Put 
	\[
	\supp(\bx):=\{i\in I\mid x_i\in A(i)-B(i)\}.  
	\]
Define the \emph{polyhedral product} $\zx$ to be the subset of $\prod_{i\in I} A(i)$ consisting of those $\bx$ such that $\supp(\bx)\in \cs(L)$.  (This terminology comes from \cite{bbcg,bbcg2}.
In \cite{densuc} it is called a ``generalized moment-angle complex.'')  If each $(A(i),B(i))$ is equal to the same pair, say $(A,B)$, then we shall write $\zl(A,B)$ instead of $\zx$.  Here is another way to define the same thing.  For each $J\in \cs(L)$, let $\cz_J(\bA,\bB):=\{\bx\in \prod A(i)\mid \supp(\bx)\le J\}$.  In other words, 
	\begin{equation*}\label{e:J}
	\cz_J(\bA,\bB)=\prod_{i\in J} A(i) \times \prod_{i\in I-J} B(i).
	\end{equation*}
(There is a slight abuse of notation here, since the coordinates in the product do not have a natural order.) 
The polyhedral product is defined by
	\begin{equation*}\label{e:polyprod}
	\zx=\bigcup_{J\in \cs(L)} \cz_J(\bA,\bB).
	\end{equation*}

\begin{example}\label{ex:dchamber}
Suppose each $A(i)$ is the unit interval $[0,1]$ and each $B(i)$ is the base point $1$.  The polyhedral product 
	\[
	K(L):=\zl([0,1],1)
	\] 
is called the \emph{chamber} associated to $L$.  
The geometric realization of the poset $\cs(L)$, 
$\flag(\cs(L))$, can be identified as a standard subdivision of $K(L)$ in such a fashion that a standard subdivision of each cube in the polyhedral product is a subcomplex of $\flag (\cs (L))$ (cf.\ \cite[Appendix A.3]{ab}, \cite[\S4.2]{bp} \cite[Ex.{} A.4.8, A.5.2]{dbook}).  
If $L$ is a triangulation of $S^{n-1}$, then $K(L)$ is an $n$-disk.  (For example, if $L$ is the boundary complex of a simplicial convex polytope, then $K(L)$ can be identified with the dual polytope.)
\end{example}

\begin{example}\label{ex:com}
Suppose each $(A(i),B(i))$ is $(D^1,S^0)$ where $D^1=[-1,1]$ and $S^0=\{\pm 1\}$.  The space $\zl (D^1,S^0)$
is featured in \cite[\S 1.2]{dbook}.  Let $\boldC_2:=\{\pm 1\}$ denote the cyclic group of order $2$.  Then $(\boldC_2)^I$ acts on the product $\prod_{i\in I} [-1,1]$  as a group generated by reflections and $\zl (D^1,S^0)$ is a $(\boldC_2)^I$-stable subspace.  Moreover, $K(L)$ is a strict fundamental domain for the action on $\zl (D^1,S^0)$ (because $[0,1]$ is a strict fundamental domain for $\boldC_2$ on $[-1,1]$).  Associated to the $1$-skeleton of $L$ there is a $\racg$, $W_L$ (See \cite[1.2]{dbook}.) It turns out that the universal cover of $\zl (D^1,S^0)$ can be identified with $\cu(W_L,K(L))$. The fundamental group of $\zl (D^1,S^0)$ is the kernel of the natural map $W_L\to (\boldC_2)^I$ (See Example~\ref{ex:trivial} and \S\ref{ss:pi1}, below.)
\end{example}

Suppose each $(A(i),B(i))$ is $(D^2, S^1)$, where $D^2$ denotes the unit disk in the complex numbers, $\bC$, and $S^1=\partial D^2$.  The space $\zl (D^2, S^1)$ is called the \emph{moment-angle complex} in \cite[Ch.\ 6]{bp}.  For $m=\ii$, the subspace $\zl(D^2, S^1)$ of $(D^2)^m \subset \bC^m$ is stable under the usual $T^m$-action, where $T^m:=(S^1)^m$.  As in the previous example, the quotient space is $K(L)$ ($=\zl([0,1],1)$.
If $L$ is a triangulation of $S^{n-1}$, then $\zl(D^2, S^1)$ is a closed $(n+m)$-manifold.  (In \cite{dj91} we were interested in ``toric manifolds'', $M^{2n}=\zl(D^2,S^1)/T^{m-n}$, for an appropriately chosen subgroup $T^{m-n}$ of $T^m$ acting freely on $\zl(D^2,S^1)$.  Some of these $M^{2n}$ are nonsingular toric varieties and are symplectic manifolds with a Hamiltonian $T^n$-action.)

Let $BS^1$ ($=\bC P^\infty$) be the classifying space for the Lie group $S^1$ and let $\xi$ be the canonical complex line bundle over 
$BS^1$.  The the total spaces of the associated $D^2$- and $S^1$-bundles are denoted $D(\xi)$ and $S(\xi)$, respectively.  So, $S(\xi)=ES^1$, the total space of the canonical principal $S^1$-bundle over $BS^1$.  Note that $D(\xi)$ is homotopy equivalent to $BS^1$ while $S(\xi)$ is contractible.  The classifying space for $T^m$ is the $m$-fold product of copies of $BS^1$.  \emph{Davis-Januszkiewicz space}, $DJ(L)$, is defined to be the Borel construction on $\zl(D^2, S^1)$, i.e.,
	\begin{equation*}\label{e:borel}
	DJ(L):=\zl(D^2, S^1) \times _{T^m} ET^m
	\end{equation*}
(cf.\ \cite{dj91} \emph{and} \cite{bbcg,bbcg2,bp,densuc}).
The space $DJ(L)$ can also be written as a polyhedral product: $DJ(L)=\zl(D(\xi),S(\xi))$ (cf.\ \cite{dj91}).  Since $D(\xi)\sim BS^1$ and $S(\xi)\sim *$ (where $\sim$ means homotopy equivalent and where $*$ is a base point), we have that $DJ(L)\sim \zl(BS^1,*)$.

\begin{example}\label{ex:realdj}(\emph{Real Davis-Januszkiewicz space}).
In a similar fashion, we can form the Borel construction on the polyhedral product in Example~\ref{ex:com}, to get the \emph{real Davis-Januszkiewicz space}, $DJ^\bR(L)$, 
	\begin{equation*}\label{e:realborel}
	DJ^\bR(L):=\zl(D^1, S^0) \times _{\boldC_2^m} E\boldC_2^m,
	\end{equation*}
where $B\boldC_2$ ($=\bR P^\infty$) is the classifying space for the cyclic $2$-group, $\boldC_2$, and $E\boldC_2$ is its universal cover.  As in the previous paragraph,
	\begin{equation}\label{e:dj2}
	DJ^\bR(L)=\zl(D(\xi),S(\xi))\sim \zl(B\boldC_2,*),
	\end{equation}
where now $D(\xi)$ means the canonical $D^1$-bundle over $B\boldC_2$ and $S(\xi)$ the $S^0$-bundle.  
\end{example}

\begin{example}\label{ex:RAAGs}
Let $1\in S^1$ be the standard base point.  Consider $\zl(S^1,1)$.  It is a subcomplex of the standard torus, $T^I$ ($=(S^1)^I$).  As we shall see in \S\ref{ss:pi1}, the fundamental group of $\zl(S^1,1)$ is the \emph{right-angled Artin group} (or $\raag$)  associated to the $1$-skeleton of $L$ (in other words, it is the graph product of infinite cyclic groups with respect to the graph $L^1$ (cf.\ Examples~\ref{ex:racg}\,(4), below). 
\end{example}

As we shall see in \S\ref{ss:pi1} and \S\ref{ss:aspprod},  $\zl(B\boldC_2,*)$ is the classifying space of the associated $\racg$ if and only if $L$ is a flag complex.  Similarly, $\zl(S^1,1)$ is the classifying space for the $\raag$ associated to $L^1$ if and only if $L$ is a flag complex (cf.\ \cite[\S3.2]{cd95} or Corollary~\ref{c:bgamma}, below).

\paragraph{Retractions.}  Suppose $I'$ is a subset of $I$ and $L'\leq L$ is the full subcomplex spanned by $I'$.  For each $i\in I-I'$ choose a base point $b_i$ in $B(i)$ and let $\boldb_{I-I'}$ be the corresponding point in $\prod_{i\in I-I'} B(i)$.  The inclusion $L'\hookrightarrow L$ induces an inclusion $i:Z_{L'}(\bA,\bB)\hookrightarrow \zl (\bA,\bB)$ by taking the product with $\boldb_{I-I'}$.

\begin{lemma}\label{l:retract}\textup{(cf.\ \cite[Lemma 2.2.3]{densuc}).}
Suppose, as above, that $L'$ is a full subcomplex of $L$.  Then $\zl(\bA,\bB)$ retracts onto $Z_{L'}(\bA,\bB)$.
\end{lemma}

\begin{proof}
The retraction $r:\zl(\bA,\bB)\to Z_{L'}(\bA,\bB)$ is induced by the projection $\prod_{i\in I} A(i) \to \prod_{i\in I'} A(i)$.
\end{proof}

\subsection{Right-angled buildings}
Given a discrete space  $E$, let $\cone(E)$ be the \emph{cone on $E$}, i.e., $\cone(E)$ is the space formed from $E\times [0,1]$ by identifying all points with second coordinate $0$.  Identify $E$ with the subspace $E\times 1$.  For example, $\cone(S^0)=D^1$. 

\begin{example}\label{ex:rank1.2} (\emph{Realizations of products of rank 1 buildings}).
As in Example~\ref{ex:rank1.1}, suppose $\{E_i\}_{i\in I}$ is a family of rank one buildings (i.e., each $E_i$ is a discrete set with more than one element).  Let $\cac=\prod_{i\in I}E_i$ be the product building.  The associated Coxeter group is $(\boldC_2)^I$ and the Davis chamber is the cube, $[0,1]^I$.  The standard realization of $\cac$ is 
\[
\cu(\cac,[0,1]^I)=\prod_{i\in I} \cone(E_i)
\]
(cf.\ \S\ref{ss:basicconst}).   The spherical realization of $\cac$ is the join of the $E_i$ (cf.\ Examples~\ref{ex:flags}\,(1)). 
\end{example}

Here is an important  generalization of Example~\ref{ex:com}.

\begin{example}\label{ex:polyrank} (\emph{Polyhedral products of rank 1 buildings}).  
Take notation as in the previous example.  Let $L$ be a simplicial complex with vertex set $I$.  Put $\Econe :=\{(\cone(E_i),E_i)\}_{i\in I}$. Then the polyhedral product of the $(\cone(E_i),E_i)$ can be identified with the $K(L)$-realization of the product building $\cac=\prod_{i\in I} E_i$, i.e.,
	\begin{equation}\label{e:zlcu}
	\cu(\cac,K(L))=\zl\Econe .
	\end{equation}
\end{example}

\begin{example}\label{ex:univ}(\emph{Universal covers}).  
We note that if the $1$-skeleton of $L$ is not a complete graph (and each $E_i$ has at least two elements), then $\zl\Econe$ will not be simply connected.  It is proved in \cite{d09} that the universal cover $\wt{\cu}$ of  $\cu(\cac,K(L))$ is the $K(L)$-realization of a building,  $\wcac$. (As a set, $\wcac$ can be identified with the inverse image in $\wt{\cu}$ of the central vertex of $K(L)$.)  The type of this building is $(W_{L^1},S)$, the $\racs$ associated to the $1$-skeleton of $L$.
\end{example}

\begin{lemma}\label{l:pilone}
The fundamental group of  $\zl\Econe$ depends only on the $1$-skeleton of $L$, i.e., 
\[
\pi_1(\zl\Econe)=\pi_1(\zlone\Econe).
\]
\end{lemma}

\begin{proof}
First consider the case where $L$ is the simplex $\gD^n$ on $\{0,1,\dots n\}$, with $n\ge2$. Then $Z_\gD\Econe=\cone E_0\times \cdots \times \cone E_n$, while $Z_{\partial \gD}\Econe$ is the join $E_0*\cdots *E_n$, which is simply connected (since $n\ge 2$).  So, $Z_{\gD^n}\Econe$ is the cone on a simply connected space.  It follows that $\zl\Econe$ can be constructed from $\zlone\Econe$ by successively adjoining cones on simply connected spaces.  So, the fundamental group does not change.

Alternatively, this can be proved by using \eqref{e:zlcu} and the arguments in \cite[Ch.\ 8, 9]{dbook}.
\end{proof}

Let $\wzl\Econe$ denote the universal cover of $\wzl\Econe$.  As a consequence of \eqref{e:zlcu} we have the following.

\begin{proposition}\label{p:RABcone}\textup{(cf.~\cite{d09}).}
Suppose $L$ is the flag complex of a simplicial graph $L^1$.  With notation as above, the standard realization of $\wcac$ is given by
\[
\cu(\wcac,K(L))=\wcu(\cac,K(L))=\wzl\Econe.
\]
\end{proposition}

\begin{lemma}\label{l:wzl}
With notation as above, $\wzl\Econe$ is contractible if and only if $L$ is a flag complex.
\end{lemma}

\begin{proof}
Let $(W,S)$ be the $\racs$ associated to $L^1$ and let $\wcac$ be the $\rab$ of type $(W,S)$ defined in the Example~\ref{ex:univ}.  By Examples~\ref{ex:polyrank} and \ref{ex:univ}, 
$\wzl\Econe=\cu(\wcac, K(L))$, the $K(L)$-realization of $\wcac$.  But the $K(L)$ realization of any $\rab$ is contractible if and only if $L$ is a flag complex.  (If $L$ is a flag complex, then the $K(L)$ realization of a $\rab$ is its standard realization which is $\cat (0)$, by \cite[Thm~11.1]{dbuild}, hence, contractible. A slightly different argument showing that  the cubical complex $\wzl\Econe$ is $\cat (0)$ goes as follows:  the link of any cubical face is a  flag complex since any such link is the join of a link in $L$ with a multiple join of discrete sets.)  If $L$ is not flag, then any $K(L)$ realization of an apartment is not contractible (\cite[Thm.\ 9.1.4, p.\ 167]{dbook}) and hence, the same is true for the building since it retracts onto any apartment.
\end{proof}

\subsection{Graph products of groups} 
Suppose we are  given a simplicial graph $L^1$ with vertex set $I$
and a family of discrete groups $\bG=\{G_i\}_{i\in I}$.

\begin{definition}\label{d:gprt}
The \emph{graph product} of the $G_i$ is the group $\gG$  formed by first taking  the free product of the $G_i$ and then taking the quotient by the normal subgroup generated by all commutators of the form $[g_i,g_j]$ where $\{i,j\}\in \edge(L^1)$, $g_i\in G_i$ and $g_j\in G_j$.
\end{definition}

\begin{examples}\label{ex:racg}\hfil

(1) If the graph $L^1$ has no edges, then $\gG$ is the free product of the $G_i$.

(2) If $L^1$ is the complete graph, then $\gG$ is the direct sum $\prod_{i\in I} G_i$ (i.e., when $I$ is finite it is the direct product).
 
(3) If all $G_i=\boldC_2$, then $\gG$ is the $\racg$ determined by $L^1$.

(4) If all $G_i=\zz$, then $\gG$ is the \emph{right-angled Artin group} (or $\raag$) determined by $L^1$. 
\end{examples}

\paragraph{Relative graph products.} Suppose that for each $i$ we are given a $G_i$-set $E_i$.  The group $G_i$ need not be effective on $E_i$.  Let $N_i$ be the normal subgroup which acts trivially.  
Put $(\bG,\bE):=\{(G_i,E_i)\}_{i\in I}$ and $\Econe:=\{(\cone E_i, E_i)\}_{i\in I}$.  As in Example~\ref{ex:polyrank}, we have the polyhedral product $\zlone\Econe$.   Let 
$\wzlone\Econe$ be the universal cover of the space $\zlone\Econe$.  The group $G:=\prod_{i\in I} G_i$ acts on  $\zlone\Econe$ and the kernel of the action is the normal subgroup $N:=\prod_{i\in I} N_i$.  Let $\gG_L$ be the group of all lifts of the $G/N$-action to $\wzl\Econe$.  In other words, there is a short exact sequence of groups,
\begin{equation*}
1\to \pi_1(\zl\Econe ) \to \gG_L \to  G/N \to 1.
\end{equation*}

Using Lemma~\ref{l:pilone}, we deduce the following.

\begin{lemma}\label{l:pilone2}
The group $\gG_L$ of all lifts of the $G/N$-action to $\wzl\Econe$ depends only on the $1$-skeleton of $L$, i.e.,
\(
\gG_L=\gG_{L^1}
\).
\end{lemma}

The group $\gG_{L^1}$ is almost the correct definition of  the ``graph product of the $G_i$ relative to the $E_i$.''  To make the definition correct we must put back in the subgroup $N$.  Thus, the \emph{graph product of the $G_i$ relative to the $E_i$}, is the group $\zlone(\bG,\bE)$, defined the extension of $\gG_{L^1}$ by $N$ given by the pullback diagram:
	\begin{equation*}
	\begin{CD}
	\zlone(\bG,\bE) @>>> \gG_{L^1}\\
	@VVV	@VVV\\
	G @>>> G/N
	\end{CD}
	\end{equation*} 

\begin{remark}
If each $G_i$ acts simply transitively on $E_i$, then $E_i\cong G_i$, $\zlone(\bG,\bE)=\gG_{L^1}$ and this group can be identified with the previously defined graph product.  If each $G_i$ is only required to be transitive, then $E_i\cong G_i/H_i$ and $\zlone(\bG,\bE)$ coincides with the graph product of the pairs $\{G_i,H_i)\}$ as defined by Januszkiewicz-\'Swi\c atkowski in \cite{js}. Discussions of ``graph products of group actions,'' along the above lines, also can be found in \cite[\S4.2]{hagdg} and \cite[\S1.4]{do10}.
\end{remark}

\begin{example}\label{ex:trivial}
Even when all the groups are trivial, the relative graph product need not be the trivial.  Indeed, suppose each $G_i$  is the trivial group and each $E_i$ is $S^0$.  Then $\zlone (\bG,\bE)=\pi_1(\zlone(D^1,S^0))$ (cf.\  Example~\ref{ex:com}).
\end{example}

\subsection{The fundamental group of a polyhedral product}\label{ss:pi1}

As before, $\ab=\{(A(i),B(i))\}_{i\in I}$ is a family of pairs of CW complexes and $L$ is a simplicial complex with vertex set $I$.   
We assume  each $A(i)$ is path connected.

Let $G_i=\pi_1(A(i))$ and let $p_i:\wa(i)\to A(i)$ be the universal cover.  Let $E_i$ denote the set of path components of $\wb(i)$, where $\wb(i):=p_i^\minus(B(i))$.  So, $E_i$  is naturally a $G_i$-set.  The collection $(\bG,\bE):=\{(G_i,E_i)\}_{i\in I}$ is the \emph{fundamental group data} associated to $\ab$.  The natural map $\wb(i)\to E_i$ which collapses each component to a point extends to a $G_i$-equivariant map $q_i:\wa(i)\to \cone E_i$.  Taking products we get a map $q:\zl \wab \to \zl\Econe$.  The universal cover $\wzl\wab\to \zl \wab$ is induced from pulling back the universal cover of $\zl\Econe$ via $q$. Using this observation together with Lemma~\ref{l:pilone}, we get the following.

\begin{lemma}\label{l:1skel}
The fundamental group of the polyhedral product depends only on the $1$-skeleton of $L$, i.e., $\pi_1(\zl\ab)=\pi_1(\zlone\ab)$.  
\end{lemma}

\begin{theorem}\label{t:pi1}
Take hypotheses as above and let $(\bG,\bE)$ be the fundamental group data associated to $\ab$.   Then the fundamental group of the polyhedral product is the graph product of the $G_i$ relative to the $E_i$, i.e., 
\[
\pi_1(\zl\ab)=\zlone (\bG,\bE).
\] 
\end{theorem}

\begin{proof}
Since $\pi_1(\zl\ab)$ is the group of deck transformations of $\wzl\wab \to \zl\ab$, it acts on $\wzl\Econe$ and hence, defines a homomorphism $\pi_1(\zl\ab)\to \gG$ which lifts to $\gf:\pi_1(\zl\ab)\to \zlone(\bG,\bE)$.  Since the universal cover $\zl\wab$ is the pullback of $\wzl\Econe$, it follows that $\gf$ is an isomorphism.
\end{proof}

\begin{example}\label{ex:realdj2} (\emph{Real Davis-Januszkiewicz space continued}).  This is a continuation of Example~\ref{ex:realdj}.  By \eqref{e:dj2}, $DJ^\bR (L)\sim\, \zl(B\boldC_2,*)$; so, $\pi_1(DJ^\bR(L))$ is the graph product of copies of $\boldC_2$ over $L^1$.  By Examples~\ref{ex:racg}\,(3),  this graph product is $W_{L^1}$, the $\racg$ associated to $L^1$.
\end{example}

\begin{example}\label{ex:gprt1} (\emph{Graph products of classifying spaces}). 
More generally, given $L^1$ and a family of groups $\bG=\{G_i\}_{i\in I}$, let $\ab=\{(BG_i,*)\}_{i\in I}$.  By Theorem~\ref{t:pi1},  $\pi_1(\zlone\ab)$ is the graph product of the $G_i$.
\end{example}

\subsection{When is a polyhedral product aspherical?}\label{ss:aspprod}
A vertex $v$ of a simplicial complex $L$ is a \emph{conelike} if it is joined by edges to all other vertices of $L$

\begin{definition}\label{d:apair}
A pair of spaces $(A,B)$ is an \emph{aspherical pair}  if 
\begin{itemize}
\item
$A$ is path connected and aspherical,
\item
each path component of $B$ is aspherical and the fundamental group of each such path component maps injectively into $\pi_1(A)$.
\end{itemize}
Given a pair $(A,B)$ with $A$ path connected, let $\wa$ denote the universal cover of $A$ and let $\wb$ be the inverse image of $B$ in $\wa$.
\end{definition}

\begin{theorem}\label{t:asphericalflag}
Suppose $I$ is the vertex set of a simpliciial complex $L$ and  $\ab=\{(A(i), B(i))\}_{i\in I}$ is a family of pairs.  Then the folllowing three conditions are necessary and sufficient for $\zl\ab$ to be aspherical.
\begin{enumeratei}
\item
Each $A(i)$ is aspherical.
\item
If a vertex $i$ is not conelike, then $(A(i), B(i))$ is an aspherical pair.
\item
$L$ is a flag complex.
\end{enumeratei}    
\end{theorem}

\begin{proof}
For simplicity, put $\zl=\zl\ab$.  
First we prove the necessity of these conditions.   Assume $\zl$ is aspherical.  Taking $L'$ to be a single vertex $i$ in Lemma~\ref{l:retract}, we get that $A(i)$ is a retract of $\zl$; hence, (i).  If the vertex $i$ is not conelike then there is another vertex $j$ which is not connected to it.  Consider the full subcomplex $L'=\{i,j\}$.  We have 
\[
\cz_{L'}:= (A(i)\times B(j)) \cup (B(i) \times A(j)).
\]
Consider the open cover of $\cz_{L'}$ by the components of $A(i)\times B(j)$ and the components of $B(i)\times A(j)$.  The nerve of this cover is a complete bipartite graph.  The set of vertices of one type is $\pi_0(B(i))$  and of the other type, $\pi_0(B(j))$.  The edges are indexed by  $\pi_0(B(i))\times \pi_0(B(j))$.  The universal cover, $\wzll$, is covered by components of $\wa(i)\times \wb(j)$ and $\wb(i)\times \wa(j)$ and it has the structure of a tree of spaces.  By Lemma~\ref{l:retract}, $\zll$ is aspherical and hence, $\wzll$ is contractible.  Since each vertex space is a retract of $\wzll$ we see that each component of $\wb(i)\times \wa(j)$ is contractible.  Since $\wa(j)$ is contractible, this means each component of $\wb(i)$ is contractible; hence, (ii).  As in \S\ref{ss:pi1} let $\{(G_i,E_i)\}_{i\in I}$ be the fundamental group data for $\ab$, in other words, $E_i=\pi_0(\wb(i))$.  
If $i$ is not a conelike vertex, then since (ii) holds, the natural map $(\wa(i),\wb(i))\to (\cone(E_i), E_i)$ is a homotopy equivalence. So, if no vertex is conelike, $\wzl(\bwa,\bwb)$ is homotopy equivalent to  $\wzl\Econe$. By Lemma~\ref{l:wzl}, $\wzl\Econe$ is contractible if and only if $L$ is a flag complex.  Hence, when there are no conelike vertices, if  $\wzl(\bwa,\bwb)$ is contractible,  then  $L$ is a flag complex.  

Next, suppose $i$ is a conelike vertex.  Let $L'$ be the full subcomplex spanned by $I-\{i\}$.  We claim $L= i*L'$. Suppose not.  Then there is a simplex $\gs$ in $L'$ such that the join $i*\partial \gs \subset L$ but the full simplex spanned by $\gs$ and $i$ is not in $L$.  Then $L'':= \gs \cup (i*\partial \gs)$ is a full subcomplex of $L$ isomorphic to the boundary of a simplex of dimension  one greater than that of $\gs$.  Note that $\cz_{L''}\Econe$ is the spherical realization of a product building (namely, $\cac''=\prod_{j\in \vertex (L'')}\ E_j$); hence, it  is homotopy equivalent to a wedge of $k$-spheres, where $k=\dim L'' +1\geq 2$.  Since each $\wt{A}_j$ is contractible, it is homotopy equivalent to $\cone (E_j)$.  One can now show that $\cz_{L''}(\bwa,\bwb)$ is $(k-1)$-connected and, by a simple spectral sequence argument, that for the first nonzero homology group we have $H_k(\cz_{L''}(\bwa, \bwb))=H_k(\cz_{L''}\Econe)$.  Hence, $\pi_k(\cz_{L''}(\bwa, \bwb)\neq 0$ for some $k\geq 2$, i.e., $\cz_{L''}(\bwa, \bwb)$ is not aspherical.  But $\cz_{L''}(\bwa, \bwb)$ is a retract of $\zl(\bwa,\bwb)$, contradicting the assumption that $\zl\ab$ is aspherical.  Hence, $L=i*L'$.
By induction on the number of vertices, we may suppose that $L'$ is flag and hence, so is $L$. 

To prove sufficiency, suppose conditions (i), (ii) and (iii) hold.  First we show that we can reduce to the case where there are no conelike vertices.  Since $L$ is flag, if there are conelike vertices, then $L$ can be written as a join $L' *\gD$, where $\gD$ is the simplex spanned by all conelike vertices.  Hence,
	\[
	\zl\ab= \cz_{L'}\ab\times \cz_\gD\ab=\cz_{L'}\ab\times \prod_{i\in \vertex(\gD)}\,A(i) 
	\]
Since each $A(i)$ is aspherical (by (i)), the product is aspherical and we are reduced to proving  that  $\cz_{L'}\ab$ is aspherical.  In other words, we can assume $L$ has no conelike vertices, i.e., that (ii) holds for all $i\in I$.  This implies that each $\wa(i)$ is contractible and that each path component of $\wb(i)$ is contractible.  Therefore, as before, $\wzl(\bwa,\bwb)$ is homotopy equivalent to $\wzl\Econe$.  By Lemma~\ref{l:wzl}, condition (iii) implies that the $\rab$, $\wzl\Econe$, is contractible and hence, so is $\wzl(\bwa,\bwb)$.
\end{proof}

\begin{corollary}\label{c:bgamma}
As in Example~\ref{ex:gprt1},  suppose we are given $L$ and a family of groups $\bG=\{G_i\}_{\i\in I}$.  Let $\gG$ denote their graph product over $L^1$ and let $\ab=\{(BG_i,*)\}_{i\in I}$.  Then $\zl\ab=B\gG$ if and only if $L$ is a flag complex.
\end{corollary}

\section{Corners}\label{s:corners}
\subsection{More definitions}
\paragraph{Nondegenerate simplicial maps and colorings.}  Suppose $L$, $L'$ are simplicial complexes with vertex sets $I$, $I'$, respectively. 

\begin{definition}\label{d:fold}
A simplicial map $L\to L'$ is \emph{nondegenerate} if its restriction to each simplex is injective.  A \emph{coloring}
is a nondegenerate map $L\to \gD$ onto a simplex.  
\end{definition}

\begin{examples}\label{ex:fold}
(1) (\emph{Barycentric subdivisions}).  Suppose $L$ is the barycentric subdivision of an $n$-dimensional convex cell complex $Y$.  Then each vertex $v$ of $L$ is the barycenter of some cell in $Y$.  Let $d(v)$ be the dimension of this cell.  Then $v\mapsto d(v)$ is a map $\vertex(L) \to I(n)$, defining a coloring $d:L\to \gD$.

(2) (\emph{Finite Coxeter complexes}).  Suppose $(W,S)$ is a spherical Coxeter system of rank $n+1$.  A fundamental simplex for $W$ on $\sphere^n$ can be identified with $\gD$ (where the mirror structure on $\gD$ is defined by letting $\gD_i$ be the codimension one face opposite to the vertex $i$.   The Coxeter complex, $\cu(W,\gD)$, is homeomorphic to $\sphere^n$ and the natural projection $\cu(W,\gD)\to \gD$ is a coloring.  A special case of this is where $W=(\boldC_2)^{I(n)}$.  In this case the Coxeter complex can be identified with the boundary complex of an $(n+1)$-dimensional octahedron.

(3) (\emph{Spherical buildings}).  More generally, if $\cac$ is a spherical building of type $(W,S)$, then its simplicial realization, $\cu(\cac,\gD)$, is a simplicial complex and the natural projection $\cu(\cac,\gD) \to \gD$ is a coloring. 

(4) (\emph{Simplices of groups}).  Suppose $\cg$ is a simplex of groups as defined in \S\ref{ss:simpcor} below.   
Let $\cu(\cg,\gD)$ be its universal cover as defined in \S\ref{ss:simpcor}.  If  $\cu(\cg,\gD)$ is a simplicial complex, then the projection $\cu(\cg,\gD)\to \gD$ is a coloring.   For example, this shows that the Deligne complex of an Artin group of spherical type is colorable (see \cite{cdjams}, \cite{deligne} or\ Examples~\ref{ex:spheretype}\,(3) below, for the definition of ``Deligne complex'').
\end{examples}

\paragraph{Properties for mirror structures.} 
Suppose  $\cm$ is a mirror structure on a space $X$ over a set $I$.  The mirror structure is \emph{flag} if its nerve $N(\cm)$ is a flag complex. Equivalently, $\cm$ is flag if for any nonempty subset $J\leq I$, the intersection $X_J$ is nonempty whenever all pairwise  intersections $X_i\cap X_j$ are nonempty for all $i$,$j\in J$. 

A space $X$ is $m$-acyclic (resp. \emph{acyclic}) if its reduced homology groups, $\oh_i(X)$, vanish for all $i\leq m$ (resp. for all $i$).
A mirror structure $\cm$ on a connected space $X$ is $0$-acyclic if each mirror is path connected as is each nonempty intersection of mirrors.  The mirror structure is \emph{acyclic} if  $X$ is acyclic as is each nonempty intersection of mirrors.


Next, we define the ``induced mirror structure'' on the universal cover $\wx$ of $X$.  Let $p:\wx \to X$ be the covering projection.  For each $i\in I$, let $E_i$ denote the set of path components of $p^\minus(X_i)$.  Let $E$ denote the disjoint union of the $E_i$. Let $\gi:E\to I$ be the natural projection which sends $E_i$ to $i$.  There is an  \emph{induced mirror structure} $\wcm=\{\wx_e\}_{e\in E}$  defined by $\wx_e:=e$. Extend $\gi$ to a natural projection $\gi:N(\wcm)\to N(\cm)$, defined by $J\mapsto \gi(J)$.

\paragraph{Corners of spaces.}
Suppose $I(n):=\{0,1,\dots, n\}$ and $\cp(I(n))$ is its power set.  A \emph{corner of spaces} is a path connected space $X$ with a mirror structure $\cm=\{X_i\}_{i\in I(n)}$ such that $X_J\neq \emptyset$ for all $J\in \cp(I(n))$.  So, the nerve of this mirror structure is the $n$-simplex, $\gD$. If $\wcm$ is the induced mirror structure on $\wx$, then $\gi:N(\wcm)\to \gD$ is a coloring.

\begin{example}\label{ex:prod1}(\emph{Products}).  
Suppose we are given the data for a polyhedral product over $\gD$, i.e., a family of pairs $\ab=\{(A(i),B(i))\}_{i\in I(n)}$.
Put $X=\cz_\gD\ab=\prod_{i\in I(n)} A(i)$.  Define a mirror structure on $X$ over $I(n))$ by letting $X_i$ be the set of those points $\bx$ such that $x_i\in B(i)$, i.e., 
	\[
	X_i\cong B(i)\times\prod_{\substack{j\\j\neq i}} A(j). 
	\]
Note that for $J\leq I(n)$, 
	\[
	X_J= \bigcap_{i\in J} X_i \cong \prod_{i\in J} B(i) \times \prod_{j\notin J} A(j).
	\]
\end{example}

\begin{example}\label{ex:cube}(\emph{An $(n+1)$-cube}).  
Suppose $\square$ is the $(n+1)$-dimensional cube $[0,1]^{n+1}$.  Define the mirror $\square_i$ by $x_i=1$.  This is the special case of the previous example where $(A(i),B(i))=([0,1],1)$.  It is also a special case of Example~\ref{ex:dchamber}.
\end{example}

\subsection{Reflection groups and asphericity}\label{ss:rga}
Suppose that $(m(i',j'))$ is a Coxeter matrix on a set $I'$, that the nerve of the corresponding Coxeter system is $L'$ and that $f:L\to L'$ is  nondegenerate simplicial map.  Define a new Coxeter matrix $(m(i,j))$ on $I$ ($=\vertex (L)$) by 
	\begin{equation}\label{e:coxmat}
	m(i,j) =
	\begin{cases}
	1, 		&\text{if $i= j$;}\\
	m(f(i),f(j)), 	&\text{if $\{i,j\}$ is an edge of $L$;}\\
	\infty, 	&\text{otherwise.}
	\end{cases}
	\end{equation}
The new Coxeter system corresponding to $(m(i,j))$ is \emph{induced} from the old one via $f$ (cf.\ \cite{d87}).

Suppose $(W,S)$ is a spherical Coxeter system where $S$ is indexed by $I(n)$ and that $X$ is a corner of spaces over $I(n)$.  When $\cu(W,X)$ is aspherical?  

\begin{theorem}\label{t:flag}
Suppose $\cm$ is a mirror structure on $X$ giving it the structure of a corner of spaces over $I(n)$ and that $(W,S)$ is a spherical Coxeter system where $S$ is also indexed by $I(n)$.  Then $\cu(W,X)$ is aspherical if and only if the following three conditions hold.
	\begin{enumeratei'}
	\item\label{i:1}
	$X$ is aspherical.
	\item
	The induced mirror structure $\wcm$ on the universal cover is acyclic.
	\item
	$\wcm$ is  flag.
	\end{enumeratei'}
\end{theorem}

\begin{proof}
Let $\ww$ be the Coxeter group induced from $W$ via $\gi:E\to I(n)$, as in \eqref{e:coxmat}.  Let $S_E=\{s_e\}_{e\in E}$ be the corresponding fundamental system of generators for $\ww$.  Let $L(\ww,S_E)$ denote the nerve of this Coxeter system.  According to \cite[p.~168]{dbook}, the natural map $\cu(\ww,\wx)\to \cu(W,X)$ defined by $[\tw,\tx]\mapsto [\gf(\tw), p(\tx]$ is the universal cover (where $\gf:\ww\to W$ is the natural epimorphism).  By definition, $\cu(W,X)$ is aspherical if and only if  $\cu(\ww, \wx)$ is contractible.  By \cite[Thm.~9.1.5, p.~167]{dbook} and \cite[Cor.~8.2.8]{dbook}, this is the case precisely when the following two conditions hold.
	\begin{enumeratei'}
	\item
	$\wx$ is contractible.
	\item\label{i:acyclic}
	For each nonempty spherical subset $J\le E$, the subcomplex $\wx_J$ is acyclic.
	\end{enumeratei'}
It follows from the assumption  that $W$ is spherical, that $L(\ww, S_E)$ is a flag complex.
Since the property of being acyclic implies the property of being nonempty, condition~\eqref{i:acyclic}$'$ entails $N(\wcm)=L(\ww,\wx)$.  So, we must have that 
\begin{itemize}
\item[(iii)$'$]
$N(\wcm)$ is also a flag complex.
\end{itemize}
\end{proof}

\begin{definition}\label{d:aspcorner}
A corner structure on $X$ is \emph{aspherical} if 
\begin{itemize}
\item
$X$ is aspherical,
\item
for all $J\in \cp(I(n))$, each path component of $X_J$ is aspherical and the inclusion of any such component into $X$ induces a monomorphism on fundamental groups.
\end{itemize}
\end{definition}
Note that if $X$ is an aspherical corner of spaces, then the induced mirror structure on $\wx$ satisfies conditions (i)$'$ and (ii)$'$ of Theorem~\ref{t:flag} (since for each spherical subset $J\le E$, the space $\wx_J$ is contractible and hence, \emph{a fortiori}, acyclic).  Condition (iii)$'$ is problematical. However, there are two nice examples when (iii)$'$ is satisfied.  The first is  the case of products
as in Example~\ref{ex:prod1} and the second is the case of Borel-Serre compactifications (as discussed in \S\ref{ss:aspflag}, below).

\subsection{Pullbacks and asphericity}\label{ss:pullback}
For any subset $J\le I(n)$, let $\jcheck$ denote its complement: 
	\begin{equation}\label{e:jcheck}
	\jcheck = I(n)-J
	\end{equation}

Suppose we are given a corner structure $\cm=\{X_i\}_{i\in I(n)}$ on a space $X$ and a coloring $f:L\to \gD^n$.  This induces a map, also denoted by $f$,  from $\cs(L)$ to $\cp (I(n))$ (sending $J$ to $f(J)$).  For each $J\in \cs(L)$,  define a certain subspace $Q(J)$ of $\cs(L) \times X$ (where $\cs(L)$ has the discrete topology) by
	\begin{equation}\label{e:pullback}
	Q(J):=(J,X_{\fjcheck}), 
	\end{equation}
and let  $Q=\bigcup_{J\in \cs(L)} Q(J)$ be the disjoint union of the $Q(J)$.  Define an equivalence relation $\sim$ on $Q$ by identifying $J'\times X_{\primecheck}$ with the corresponding subspace of $J\times X_{\fjcheck}$, whenever $J'\leq J$.  By definition, the \emph{pullback}, $f^*(X)$, is the quotient space $Q/\sim$\,.  

\begin{example}\label{ex:KL}
Suppose, as in Example~\ref{ex:cube}, that $X$ is the cube $[0,1]^{n+1}$.  Then $f^*(L)$ can be identified with $K(L)$, the chamber of $L$ defined in Example~\ref{ex:dchamber}.
\end{example}

\begin{theorem}\label{t:asphericalcorner}
Suppose the corner structure $\cm$ on $X$ satisfies the following conditions.
\begin{enumeratei'}
\item
$X$ is aspherical.
\item
The induced mirror structure $\wcm$ on the universal cover $\wx$ is acyclic.
\item
$\wcm$ is  flag.
\end{enumeratei'}
If $L$ is also a flag complex, then $f^*(X)$ is aspherical.
\end{theorem}

The induced mirror structure on $\wx$ is indexed by $E$.  Using $\gi:E\to I(n)$ we get a mirror structure over $I(n)$ on $\wx$ defined by 
	\begin{equation}\label{e:wx}
	\wx_i=\coprod _{e\in \gi^\minus (i)} \wx_e.
	\end{equation}
Since $f^*(\wx)$ is clearly a covering space of $f^*(X)$, it is equivalent to prove that $f^*(\wx)$ is aspherical.  We will prove this by constructing an auxiliary cubical complex $Y$, homotopy equivalent to $f^*(\wx)$, and then use the theory of polyhedral nonpositive curvature to show that $Y$ is aspherical.

To simplify notation let $N$ be $N(\wcm)$, the nerve of $\wcm$, and  
put $K=K(N)$.  $K$ also has a mirror structure over $I(n)$ defined as in \eqref{e:wx} by $K_i=
\coprod _{e\in \gi^\minus (i)} K_e$.  Our complex $Y$ is $f^*(K)$.  As explained in Example~\ref{ex:dchamber}, the space $K=\cz_N([0,1],1)$ is a cubical complex.  Since $f^*(K)$ is constructed by pasting together copies of $K$ along cubical subcomplexes, we see that $Y$ is also a cubical complex.  

\begin{lemma}\label{l:link}
The link of any cubical face of $Y$ is isomorphic to the join of a link of a simplex in $L$ with the link of a simplex in $N$ (in both cases the empty simplex is allowed).
\end{lemma}

\begin{proof}
Given a closed cubical face $F$ of $K=\cz_N([0,1],1)$, put
\begin{align*}
Z(F)&=\{e\in E\mid x_e\vert_F = 0\}\\
D(F)&=\{e\in E\mid x_e\vert_F = 1\}\\
\supp(F)&=E-D(F)\in \cs(N).
\end{align*}
A cubical face $F'$ of $Y$ is a pair $(J,F)$ where  $J\in \cs(L)$ and $F$ is a cubical face of $K$ with $f(J)=\gi(Z(F))\in \cp(I(n))$.  Moreover, 
\[
\Lk(F',Y)=\Lk(J,L) * \Lk(\supp(F),N).
\]
\end{proof}

\begin{proof}[Proof of Theorem~\ref{t:asphericalcorner}]
The proof consists of two parts.  First of all, we show that the cubical complex $Y$ is locally $\cat (0)$; hence, aspherical.  Secondly, we show that a natural map $\gf: f^*(\wx)\to Y$ is a homotopy equivalence.

By Gromov's Lemma (cf.\ \cite[Thm.~5.18, p.\,211]{bh} or \cite[Appendix I.6]{dbook}), the piecewise Euclidean metric on a cubical complex is locally $\cat(0)$ if and only if the link of each vertex is a flag complex.  It is easy to see that the link of any simplex in a flag complex is again a flag complex.  So, by Lemma~\ref{l:link}, $Y$ is locally $\cat(0)$.  (However, it is generally not simply connected since the intersection of various copies of $K$ in $Y$ need not be connected.)

Since, for each $J\in \cs(N)$, $K_J$ is a cone, there is a natural map $\gf: \wx \to K$, sending $\wx_J$ to $K_J$.
By condition (ii)$'$, for each $J\in \cs(N)$, the intersection of mirrors $\wx_J$ is acyclic.  It follows that, for each $J\in \cs(N)$, the union $X^J$ is acyclic.  In particular, each such $X^J$ is path connected.  Since $f^*(K)$ and $f^*(\wx)$ are both covered by contractible pieces homeomorphic to $K$ or $\wx$ respectively, and since these pieces are glued together in the same combinatorial fashion along the same number of components, it follows from van Kampen's Theorem that $\pi_1(f^*(\wx))\cong \pi_1(f^*(K))$ ($=\pi_1(Y)$).  Let $U$ denote the universal cover of $f^*(\wx)$ (this is also the universal of $f^*(X)$) and let $\wy$ denote the universal cover of $Y$.  Let $\tilde{\gf}:U\to \wy$ be the map covering $\gf$.  Since $U$ and $\wy$ are both unions of contractible pieces which are glued together along acyclic subcomplexes in the same combinatorial pattern, it follows that $\tilde{\gf}$ is a homotopy equivalence. By the previous paragraph $\wy$ is contractible; hence, so is $U$.
\end{proof}

\begin{remark}\label{r:310}
Suppose that $L$ is the Coxeter complex of a spherical Coxeter system $(W,S)$ where $S$ is indexed by $I(n)$ and let $X$ be a corner of spaces over $I(n)$.  By Examples~\ref{ex:fold}\,(2), we can choose a coloring $f:L\to \gD$.  Then $f^*(X)\cong \cu(W,X)$.  Hence, the reflection group trick of \S\ref{ss:rga} is a special case of the pullback construction.
\end{remark}

\subsection{Simplices of groups and corners of groups}\label{ss:simpcor}
A poset $\cp$ can be regarded as a category where there is a morphism from $p$ to $q$ if and only if $p\le q$.  
 A \emph{poset of spaces over $\cp$} is a functor from $\cp$ to the category of spaces and inclusions.  (We require the inclusions to be cofibrations.)  Similarly, a \emph{poset of groups} is a functor from $\cp$ to the category of groups and monomorphisms.

The \emph{colimit} of a poset of spaces $\{X(p)\}_{p\in \cp}$ is the union,
	\begin{equation*}\label{e:union}
	X:=\bigcup_{p\in \cp} X(p)
	\end{equation*}
Similarly, given a poset of groups $\{G(p)\}$ we can form the colimit, $\varinjlim G(p)$.  Of course, if $\cp$ has a final element $q$, then $X=X(q)$ and $G=G(q)$.

A \emph{(simple) simplex of groups} is a poset of groups over $\cp(I(n))_{<I(n)}$.
Thus, $\cg$ associates to each proper subset $J<I(n)$ a group $G_J$ and to each pair $(J',J)$ with $J'\leq J$ a monomorphism $G_{J'}\hookrightarrow G_J$.  The colimit of these groups is the \emph{fundamental group of $\cg$}, denoted $\pi_1(\cg)$.  (See Chapter III.$\cac$ of \cite{bh}.)  $\cg$ is \emph{developable}  if, for each $J< I(n)$, the natural homomorphism $G_J\to \pi_1(\cg)$ is injective.  One way in which a developable simplex of groups can arise is from a cell-preserving action of a group $G$ on a simplicial cell complex with strict fundamental domain an $n$-simplex $\gD$ with its codimension one faces indexed by elements of $I(n)$.  (For us, a ``simplicial cell complex'' will not necessarily be a simplicial complex --  it means a space formed by gluing together simplices along unions of faces; however,  the intersection of two simplices is not required to be a common face.)  The group $G_J$ is then the isotropy subgroup at $\gD_J$.  Conversely, as we shall explain in the next paragraph, given any developable simplex of groups $\cg$, there is an action of $G=\pi_1(\cg)$ on a simply connected simplicial cell complex $U$ with strict fundamental domain $\gD$ so that the isotropy group at $\gD_J$ is  $G_J$.  $U$ is called the \emph{universal cover} of $\cg$.  (Even when $\cg$ is not developable a universal cover can still be defined; however, the isotropy group at $\gD_J$ might  only be a quotient of $G_J$.)

Similarly, a \emph{corner of groups} $\cg$ is a functor from $\cp(I(n))$ to the category of groups and monomorphisms.  Since $I(n)$ is a terminal object of $\cp(I(n))$, the colimit of the $G_J$ is $G_{I(n)}$ ($=\pi_1(\cg)$).  So, a  corner of groups is automatically developable. 
If $\cg$ is any corner of groups, then we get a developable simplex of groups, $\cg_0$, by deleting the terminal object.  Moreover, $G_{I(n)}$ is a quotient of $\pi_1(\cg_0)$. Given a corner of groups $\cg$, put $G=G_{I(n)}$.  

Next, we need to define a corner structure on the $(n+1)$-cube, different from the one in Example~\ref{ex:cube}, which was used above.  To emphasize the difference we will regard the new one as giving the cube the structure of a poset of spaces over $\cp(I(n))^{\mathrm{op}}$ (the power set of $I(n)$ ordered by reverse inclusion).  

\begin{example}\label{ex:posetcube}(\emph{The cube again,} cf.\  Example~\ref{ex:cube}). 
For each $J\le I(n)$ let $\square(J)$ be the face of $\square$ defined by
	\begin{equation*}\label{e:squarej}
	\square(J):=\{x\in \square\mid x_i=0 \text{ for all } i\in J\}.
	\end{equation*}
(The difference from Example~\ref{ex:cube} is that the mirrors are defined by $x_i=0$ rather than $x_i=1$.) 
Note that $\flag(\cp(I(n))$ is a standard subdivision of $\square$ and $\flag(\cp(I(n))_{\le \jcheck}$ is a standard subdivision of $\square(J)$, cf.\ \cite[Appendix A.3]{ab}, \cite[\S4.2]{bp}, or \cite[Ex.~A.2.5]{dbook}. 
\end{example}

Using the new mirror structure on $\square$,  a version of the basic construction gives a cubical complex, $\cu(G/G_\emptyset,\square)$, defined as follows.  As in \eqref{e:Ix}, given $x\in \square$, put $I(x):=\{i\in I(n) \mid x_i=0\}$.  Define an equivalence relation $\sim$ on $G/G_\emptyset \times \square$ by 
	\[
	(gG_\emptyset,x)\sim (g'G_\emptyset,x') \iff x=x' \text{ and } gG_{I(x)}=g'G_{I(x)}, 
	\]
and as in \eqref{e:basic}, put 
	\begin{equation}\label{e:empty}
	\cu(G/G_\emptyset ,\square)=(G/G_\emptyset \times \square)/\sim.  
	\end{equation}
It is the universal cover of $\cg$ as a complex of groups.  If $\gD$ denotes the $n$-dimensional simplex with its usual mirror structure, then one can define $\cu(G/G_\emptyset ,\gD)$ similarly.  Call it the \emph{link} of $\cg$.  The universal cover of $\cu(G/G_\emptyset ,\gD)$ as a space, denoted by $\cu(\cg_0,\gD)$, is the universal cover of $\cg_0$ as a complex of groups.

\begin{examples}\label{ex:spheretype} Suppose $(W,S)$ is a spherical Coxeter system with $S$ indexed by $I(n)$. This gives rise to various corners of groups.

(1) First, it gives the data for corner of groups $\cw$ which assigns $W_J$ to $J$.  Its link is the Coxeter complex $\cu(W,\gD)$. (This is the universal cover of associated simplex of groups $\cw_0$ when  $n>1$.)  The Coxeter complex is a triangulation of $\sphere ^n$.  

(2) Suppose $G$ is a chamber-transitive automorphism group of a spherical building $\cac$ of type $(W,S)$.  Fix a chamber $C\in \cac$ and let $G_J$ denote the stabilizer of the $J$-residue containing $C$.  This gives us a corner of groups $\cg$ with $G_{I(n)}=G$.  The link of $\cg$ is the classical realization, $\cu(\cac,\gD)$.

(3) Let $A$ ($=A_{I(n)}$) be the Artin group of spherical type associated to $(W,S)$.  In a similar fashion to (1), $A$ gives rise to the \emph{corner of Artin groups} $\ca$ which associates to a subset $J\le I(n)$ the corresponding Artin group $A_J$.  In this case, the link of $\ca$ is  the \emph{Deligne complex}.

Actually in all three examples the assumption that $W$ is spherical is unnecessary.
\end{examples}

\begin{nonexample}\label{nonexample}
Suppose that $(W,S)$ is a finite Coxeter system of rank $n+1\ge 3$.  Let $\gD$ be a fundamental simplex so that the Coxeter complex, $\cu(W,\gD)$, is a triangulation of $\sphere^n$.  Let $W^+$ be the orientation-preserving subgroup.  Color the $n$-simplices which are translates of $\gD$ by an element of $W^+$ white and color the others black.  Call the union of the white simplices $\gL$.  (It is the complement in $\sphere ^n$ of the interiors of the black simplices.)  The group $W^+$ acts on $\gL$ with $\gD$ a strict fundamental domain.  It acts on the $\cone \gL$ as well.  Hence we get a corner of groups with link $\gL$.  Since it has ``missing simplices'', $\gL$ is obviously not a flag complex.  If $n+1> 3$, then $\gL$ is simply connected and hence, is the universal cover of the corresponding simplex of groups.  So,  universal covers  of simplices of groups need not be flag complexes.
\end{nonexample}

\begin{example}\label{ex:connected} (\emph{The associated corner of groups}).
Next suppose $\cm=\{X_i\}_{i\in I(n)}$ is a corner structure on $X$ and that $\cm$ is $0$-acyclic.  Put $G_{I(n)}= \pi_1(X)$ and let $G_J$ be the image of $\pi_1(X_{\jcheck})$ in $G_{I(n)}$, where $\jcheck$ is defined in \eqref{e:jcheck}. (Here we want to choose the base point in $X_{I(n)}$.)  This defines a corner of groups $\cg(\cm)$, called the \emph{associated corner of groups}.  Consider the induced mirror structure $\wcm=\{\wx_e\}_{e\in E}$ on the universal cover $\wx$ and let $N(\wcm)$ be its nerve.  The group $G_{I(n)}$ acts on $N(\wcm)$.  If $\wcm$ is also assumed to be $0$-acyclic, then any $n$-simplex of $N(\wcm)$ is a strict fundamental domain.  It follows that $N(\wcm)=\cu(G/G_\emptyset, \gD)$, the link of the associated corner of groups.
\end{example}

In the next example we show that any corner of groups can occur as $N(\wcm)$ for a corner structure on some space $X$.

\begin{example}\label{ex:gpnm}
Suppose $\cg$ is a corner of groups and $G=G_{I(n)}$.  Let $Y$ be any space with $\pi_1(Y)=G$ and with universal cover $\wy$.  Put 
	\[
	X:= \wy \times _G\, \cu(G/G_\emptyset,\square),
	\]
where $\cu(G/G_\emptyset,\square)$ is defined by \eqref{e:empty}.  
Since $\cu(G/G_\emptyset,\square)$ is contractible (it is a cone), by taking projection onto the first factor, we see that $X$ is homotopy equivalent to $Y$.  In particular, if $Y=BG$, then $X$ is also a model for the classifying space of $G$.
Projection onto the second factor induces a map $p:X\to \cu(G/G_\emptyset,\square)/G=\square$.  Put 
	\[
	X_J:= p^\minus (\square (J)) \quad\text{and}\quad X_i:= X_{\{i\}},
	\]
so that $\cm=\{X_i\}_{i\in I(n)}$ is a corner structure on $X$. Since $\square(J)$ is the union of simplices in $\flag(\cp(I(n))_{\le \jcheck}$ with maximum vertex $\jcheck$,  $\cu(G/G_\emptyset,\square(J))$ is homotopy equivalent to $G/G_J$.  So, $p^\minus (\square(J))$ is homotopy equivalent to $\wy \times_G (G/G_J) = \wy/G_J$.  Thus, $\pi_1(X_J)=G_J$.

As a further example we could let $Y=BG$ where, as in Examples~\ref{ex:spheretype}, $G$ is either a spherical Coxeter group $W$ or a spherical Artin group of rank $n+1$.  Then
\begin{align}
BW \sim\, & EW\times_W \cu(W,\square),\label{e:W}\\
BA \sim \,& EA\times_A \cu(A,\square).\label{e:A}
\end{align}
(Formula \eqref{e:A} follows from \cite[Thm.\,1.5.1, p.\,607]{cdjams} and Deligne's solution of the $K(\pi,1)$ Problem for spherical Artin groups in \cite{deligne}.)
\end{example}

\paragraph{Corners of $G$-sets.} 
The hypothesis in Example~\ref{ex:connected} that the corner structure on $X$ is $0$-acyclic seems unnatural.  To get around this we need to talk about corners of $G$-sets rather than corners of groups.  
Suppose $G$ is a group. A \emph{poset of $G$-sets} is a cofunctor from a poset $\cp$ to the category of $G$-sets and $G$-equivariant maps.
For example, given a poset of groups $\{G(p)\}_{p\in \cp}$ with a final object $G=G(q)$, we get a poset of $G$-sets $\{G/G(p)\}_{p\in \cp}$.  
A \emph{corner of $G$-sets} is a cofunctor from $\cp(I(n))$ to $G$-sets.
Suppose $\cE$ is a corner of $G$-sets.  In other words, for each $J\le I(n)$ we are given a $G$-set $E_J$ and whenever $J\le J'$ a $G$-equivariant map $\gf_{J,J'}:E_{J'} \to E_J$.  Put $\gf_J=\gf_{J,I(n)}:E_{I(n)}\to E_J$.   As before, put $\cu(\cE, \square)= (E_{I(n)}\times \square)/\sim$, where $\sim$ is defined by:  
	\[
	(e,x)\sim (e',x') \iff x=x' \text{ and } \gf_{I(x)}(e)=\gf_{I(x)}(e'). 
	\]
The space $\cu(\cE,\square)$ is the \emph{universal cover of the corner of $G$-sets $\cE$}. 
The simplicial cell complex $\cu(\cE,\gD)$ is defined similarly.  Associated to any corner structure $\cm$ on $X$, we have a corner of $G$-sets $\cE(\cm)$ as in Example~\ref{ex:connected} and as in that example we get the following.

\begin{proposition}\label{p:Gset}\textup{(cf.~Example~\ref{ex:connected}).}
Suppose $\cm$ is a mirror structure on $X$ and $\wcm$ is the induced structure on the universal cover. 
Then the nerve of $\wcm$ satisfies $N(\wcm)=\cu(\cE(\cm), \gD)$.  Hence, $K(N(\wcm))=\cu(\cE,\square)$.
\end{proposition}

\subsection{Examples of aspherical corners satistfying the flag condition}\label{ss:aspflag}
When is $N(\wcm)$ a flag complex?  For us, the most common reason for this to be true will be that it is a spherical building (cf.\ Examples~\ref{ex:flags}\,(2)).

\paragraph{Products.}
Suppose $\ab=\{(A(i),B(i))\}_{i\in I(n)}$ and as in Example~\ref{ex:prod1},  $X:=\cz_\gD\ab=\prod_{i\in I(n)}A(i)$ is the corresponding corner, where $X_i$ is the set of points with $i^{th}$-coordinate lying in $B(i)$.  Put $G(i)=\pi_1(A(i))$ and $G=\pi_1(X)=G(0)\times \cdots \times G(n)$.  Let $p:\wx \to X$ be the universal cover and $E_i$ the $G$-set of path components of $p^\minus(X_i)$.  Let $p(i):\wt{A(i)}\to A(i)$ be the universal cover and let $E(i)$ be the $G(i)$-set of path components of $p(i)^\minus(B(i))$.  Since each $A(j)$ is connected, the $G$-set $E_i$ is identified with the $G(i)$-set $E(i)$, equivariantly with respect to the projection $G\to G(i)$.  (For example, if each $B(i)$ is connected and $H(i):=\pi_1(B(i))$, then $G_i:=\pi_1(X_i)=G(0)\times\cdots \times H(i)\times \cdots \times G(n)$ and
	\[
	E_i=G/G_i \cong G(i)/H(i) = E(i).\text{)}
	\]
As at the end of the previous subsection, we get a corner of $G$-sets $\cE$, with $E_J\cong \prod_{i\in J} E(i)$, for $J\in \cp(I(n))$.  It follows that the simplicial cell complex $\cu(\cE,\gD)$ is the join $E(0)*\cdots\ast E(n)$.  In other words, $\cu(\cE,\gD)$ is the spherical realization of a product of rank one buildings, $E(0)\times \cdots \times E(n)$ (or the cone on such a spherical realization); hence, a flag complex (cf.\ Examples~\ref{ex:flags}\ (1) and (2)).  So, by Proposition~\ref{p:Gset}, $\wcm$ is flag.  Suppose each pair $(A(i),B(i))$ is an aspherical pair (cf.\ Definition~\ref{d:apair}). Then since each component of $X_J$ is aspherical, each component of its inverse image in $\wx$ is contractible.  Thus, $N(\wcm)$ is acyclic.  So, we have proved the following.

\begin{theorem}\label{t:prodflag}
Suppose $\{(A(i),B(i))\}_{i\in I(n)}$ is a family of aspherical pairs and that $X=A(0)\times\cdots\times A(n)$ has corner structure $\cm$ as above.  Then $X$ and $\cm$ satisfy conditions \textup{(i)$'$}, \textup{(ii)$'$}, \textup{(iii)$'$} of Theorems~\ref{t:flag} and \ref{t:asphericalcorner}.
\end{theorem}

\paragraph{Borel-Serre compactifcations.}
Suppose $\bG$ is a semisimple linear algebraic group defined over $\qq$.  Let $\bG(\rr)$ and $\bG(\qq)$ denote its real and rational points , respectively.  Suppose $\bG$ has $\qq$ rank $n+1$.  Let $K$ be the maximal compact subgroup of $\bG(\rr)$ and $D$ denote the symmetric space $K\backslash \bG(\rr)$.  In \cite{bs} Borel and Serre introduced a bordification $\od$ of $D$ such that
\begin{itemize}
\item
$\od$ is a manifold with corners (see \S\ref{ss:corman}, below).
\item
$\od$ is $\bG(\qq)$-stable.
\item
If $\gG< \bG(\qq)$ is any arithmetic lattice in $\bG(\rr)$, then $\od/\gG$ is compact.
\end{itemize}
If $\gG$ is torsion-free, then $M:=K\backslash \bG(\rr)/\gG$ is a smooth manifold with corners, called the \emph{Borel-Serre compactification} of $D/\gG$.  Associated to a $\qq$-split torus there is a spherical Coxeter system $(W',S')$ of rank $n+1$.  Index $S'$ by $I(n)$.  The strata of $M$ (or of $\od$) correspond to subsets of $I(n)$.  Let $M_J$ denote the stratum correponding to $(W'_{\jcheck}, S'_{\jcheck})$, where $\jcheck $ is defined by \eqref{e:jcheck}. The codimension one strata are the $M_i$ ($=M_{\{i\}}$).  This gives $M$ a corner structure, $\cm=\{M_i\}_{i\in I(n)}$.  (In fact, $M$ is a ``corner of manifolds'' as in Definition~\ref{d:mfldcor}, below.)  For us, the most salient feature of this mirror structure is the following
\begin{itemize}
\item
The the nerve of induced mirror structure $\wcm$ on $\od$  can be identified with the spherical building for $\bG(\qq)$.  
\end{itemize}
So, $N(\wcm)$ is a flag complex (cf.\ Examples~\ref{ex:flags}\,(2)).
The strata have the following structure:
\begin{itemize}
\item
Each $M_J$ is a manifold with corners.  
\item
The interior of $M_J$ is a fiber bundle over a locally symmetric space with fiber a nilmanifold.  
\item
In particular, each $M_J$ is aspherical and the map $\pi_1(M_J)\to \pi_1(M)=\gG$ is injective.
\end{itemize}
(See \cite[pp. 637--640]{ab}, \cite{bs},\cite{goresky}, \cite{l}.)  So, we have proved the following. 

\begin{theorem}\label{t:bsflag}
Suppose $M$ is the Borel-Serre compactification of $D/\gG$, where $\gG$ is a torsion-free, arithmetic group of $\qq$-rank $n+1$ and where $M$ has corner structure $\cm$ as above.  Then $M$ and $\cm$ satisfy conditions \textup{(i)$'$}, \textup{(ii)$'$}, \textup{(iii)$'$} of Theorems~\ref{t:flag} and \ref{t:asphericalcorner}.
\end{theorem}

In the case of the Borel-Serre compactification,  the fact that the reflection group trick yields an aspherical manifold $\cu(W,M)$ (i.e., that the conclusion of Theorem~\ref{t:flag} holds) has been proved previously by Phan \cite{phan}.

\begin{remark}
Suppose that, for $i\in I(n)$, $\gG(i)$ is a torsion-free  arithmetic group of rank 1,  that $D(i)/\gG(i)$ is the corresponding locally symmetric space and that $M(i)$ is its Borel-Serre compactification.  Then $\gG:=\gG(0)\times \cdots \times \gG(n)$ is an arithmetic group  of rank $n+1$ and $M=M(0)\times \cdots \times M(n)$ is the Borel-Serre compactication of its locally symmetric space.  Moreover, $(M(i),\partial M(i))$ is an aspherical manifold with $\pi_1$-injective, aspherical boundary.  So, in this case, the Borel-Serre compactification is a special case of the products discussed at the beginning of this subsection. 
\end{remark}

\paragraph{More examples.}  In Examples~\ref{ex:flags} and \ref{ex:fold} we gave some examples of colorable flag complexes, e.g., 
\begin{itemize}
\item
triangulations of $S^n$ arising from simplicial arrangements in $\rr^{n+1}$ (Examples~\ref{ex:flags}\, (3)) and Examples~\ref{ex:simple}\, (1) below, 
\item
barycentric subdivisions, i.e., flag complexes of the poset of cells in cell complexes (Examples~\ref{ex:fold}\,(1)).
\end{itemize}

\begin{examples}\label{ex:simple}
(\emph{Some simple polytopes}).  We discuss three types of simple polytopes, which are dual to colorable flag complexes.  

(1) (\emph{Zonotopes}). Suppose $\gL$ is the boundary complex of the simplicial polytope associated to some simplicial hyperplane arrangement in $\rr^{n+1}$.   Its dual polytope $P$ is a \emph{zonotope} (a simple zonotope since $\gL$ is a simplicial complex).  Let $\cf (P)$ be the set of codimension one faces of $P$.  It can be seen that $\gL$ is colorable (cf.\ \cite[Lemma~4.2.6]{djs}).  A coloring $f:\gL\to \gD$ induces a function also denoted by $f$ from $\cf$ to $I(n)$.  Define a mirror structure $\{P_i\}_{i\in I(n)}$ on $P$ by
\[
P_i= \bigcup_{F\in f^\minus(i)} F.
\]
For example, if we start with the coordinate hyperplane arrangement, then $\gL$ is the boundary complex of a $(n+1)$ dimensional octahedron and $P$ is the cube $\square= \cz_\gD(D^1, S^0)$ where $\square _i$ is defined by $x_i=\pm 1$

(2) (\emph{Duals of barycentric subdivisions}).  Suppose $\gL$ is the barycentric subdivision of the boundary complex of a convex polytope in $\rr^{n+1}$ and that $P$ is the dual polytope to $\gL$.  The coloring $d$ of Examples~\ref{ex:fold}\,(1) induces $d:\cf(P)\to I(n)$ and we define $P_i$ as before.

(3) (\emph{Coxeter zonotopes}).  Suppose $(W,S_{I(n)})$ is a spherical Coxeter system of rank $(n+1)$.  Then $W$ has a representation as an orthogonal reflection group on $\rr^{n+1}$ and the corresponding triangulation $\gL$ of $S^n$ is the Coxeter complex, $\cu(W,\gD)$. The dual zonotope $P$ is called the \emph{Coxeter polytope} in \cite{dbook}.  We have that $P\cong \cu(W,\square)$, which is the universal cover for the corner of groups $\cw$ discussed in  Examples~\ref{ex:spheretype}\,(1).  So, as in \eqref{e:W}, we have the aspherical corner $EW\times _W P$.

In all three cases, conditions (i)$'$, (ii)$'$, (iii)$'$ of Theorems~\ref{t:flag} and \ref{t:asphericalcorner} are satisfied.  The aspherical manifolds and spaces which result from applying these theorems were discussed previously in \cite{d87}.
\end{examples}

\begin{example}\label{ex:sphericalart} 
(\emph{Artin groups of spherical type}).  Suppose, as in Examples~\ref{ex:simple}\, (3), that $(W,S)$ is a spherical Coxeter group of rank $n+1$.  Let $A$ be the associated Artin group.   Salvetti defined a cell complex $X'$ which was homotopy equivalent to the complement of the reflection hyperplane arrangement in  $\cc^{n+1}$.  (The fundamental group of $X'$ is the associated pure Artin group.) The quotient $X:=X'/W$ is called the \emph{Salvetti complex}.  Its fundamental group is $A$.  The CW complex $X$ can be formed by identifying faces of the same type in the Coxeter zonotope $P$ associated to $(W,S)$ (cf.\ \cite{cd95}).  (When $W$ is right-angled and $L$ is the associated flag complex, $X$ is the polyhedral product $\zl(S^1,1)$ of Example~\ref{ex:RAAGs}.  This defines a mirror structure $\cm=\{X_i\}_{i\in I(n)}$, where $X_i$ is the image of $P_i$ in $X$ (cf.\ Example~\ref{ex:simple}\ (1)).  The associated corner of groups is the corner of groups $\ca$ defined in Example~\ref{ex:spheretype}\,(3). Its link is the Deligne complex $\cu(A,\gD)$.  As mentioned previously in Examples~\ref{ex:flags}\, (4), it was proved in \cite{cdjams} that $\cu(A,\gD)$ is a flag complex. The CW complexes $X$ and $EA\times _A\, \cu(A,\square)$ are homotopy equivalent (as corners of spaces).  It follows that $X$ is an aspherical corner, satisfying conditions (i)$'$, (ii)$'$, (iii)$'$ of Theorems~\ref{t:flag} and \ref{t:asphericalcorner}.   
\end{example}

\section{Closed aspherical manifolds}\label{s:closed}
\subsection{Corners of manifolds}\label{ss:corman}
An $n$-dimensional smooth \emph{manifold with corners} is a second countable Hausdorff space $M$ which is differentiably locally modeled on $[0,\infty)^n$.  If $\gf:U\to [0,\infty)^n$ is any coordinate chart, then the number of coordinates of $\gf(x)$ which are equal to $0$ is independent of the chart and is denoted by $c(x)$.  An \emph{(open) stratum} of codimension $k$ is a component of $\{x\in M\mid c(x)=k\}$.  A \emph{stratum} is the closure of such a component.  As in \cite[p.~180]{dbook} one can define a topological manifold with corners in a similar fashion.

For example, 
a simple polytope is a manifold with corners.

\begin{definition}\label{d:mfldcor}
Suppose $M$ is a manifold with corners.
A corner structure $\{M_i\}_{i\in I(n)}$ on $M$ is a \emph{corner of manifolds} if  each $M_J$ is a union of closed strata of codimension $|J|$.
\end{definition}

\begin{examples}\label{ex:corman}(\emph{Corners of manifolds}).

(1) (\emph{Products of manifolds with boundary}). Suppose 
$\{(M(i),\partial M(i))\}_{i\in I(n)}$ is a  collection of manifolds with boundary.  Then $M:=\prod_{i\in I(n)} M(i)$ is a manifold with corners.  As in Example~\ref{ex:prod1}, a corner structure on $M$ is defined by 
\[
M_i=M(0)\times \cdots\times \partial M(i)\times \cdots \times M(n)\cong \partial M(i) \times \prod_{j, j\neq i} M(j).    
\]
This gives $M$ the structure of a corner of manifolds.  
A point $\bx \in M$  lies in an (open) stratum of codimension $k$ if $\{i\mid x_i\in \partial M(i)\}$ has precisely $k$ elements.

(2) (\emph{Sequence of fiber bundles with manifolds with boundary as fibers}).  This is a generalization of the previous example.  Suppose we have a sequence of fiber bundles $\pi_i:X(i)\to X(i-1)$, with $-1\le i \le n$, where $X(-1)$ is a point and where the fiber of $\pi_i$ is a manifold with boundary, $M(i)$.  Then the points of $X(i)$ lying in the boundaries of the fibers form a subbundle with fiber $\partial M(i)$, which we denote $\partial_i X(i)\to X(i-1)$.  Put $\partial_0 X_0= \partial M(0)$.  Suppose $i>0$ and suppose by induction that  we have defined $\partial_j X(i-1)$ for all $j\le i-1$.  Put $\partial_j(X(i)):=\pi_i^\minus (\partial_j X(i-1))$.  It is clear that $X(i)$ is a manifold with corners locally modeled on $M(0)\times\cdots \times M(i)$.  Moreover, if we set $X=X(n)$ and define a mirror structure $\{X_i\}_{i\in I(n)}$ by $X_i=\partial_i X$, then $X$ is a corner of manifolds.

(3) (\emph{Borel-Serre compactifications}).  See \S\ref{ss:aspflag}.

(4) (\emph{Certain hyperbolic manifolds with corners}).  These examples are related to the ``strict hyperbolization procedures'' of \cite{cdstrict}.  According to \cite[Cor.\ 6.2]{cdstrict}, for each  $n$, one can construct a hyperbolic $(n+1)$-manifold with corners $X$ with corner  structure $\cm=\{X_i\}_{i\in I(n)}$ such that 
\begin{itemize}
\item
Each $X_i$ has two components, each of which is a  codimension one stratum.  (Moreover, if we regard each $X_i$ as consisting of two mirrors, then the nerve of the resulting mirror structure is the boundary complex of an $(n+1)$-octahedron.)
\item
Each component of $X_i$ is totally geodesic and these components intersect orthogonally.
\end{itemize}
Let $\wx$ be the universal cover.  We claim that the nerve $N=N(\wcm)$ of the induced mirror structure on $\wx$ is a flag complex.  The argument is somewhat indirect.  Apply the right-angled reflection group trick to $X$ to get a closed hyperbolic manifold $\cu((\boldC_2)^{I(n)}, X)$.  Its universal cover $\cu(W_N,\wx)$ can be identified with hyperbolic $(n+1)$-space, which is contractible.  (Here $W_N$ is the $\racg$ associated to the $1$-skeleton of $N$.) As we pointed out in the proof of Theorem~\ref{t:flag}, a necessary condition for $\cu(W_N,X)$ to be acyclic is that $N$ be a flag complex (cf.\ \cite[Cor.~8.2.8]{dbook}).  Hence, $N$ is a flag complex.  So, $X$ and $\cm$ satisfy conditions (i)$'$, (ii)$'$, (iii)$'$ of Theorems~\ref{t:flag} and \ref{t:asphericalcorner}.  If $f:L\to \gD$ is a coloring of a flag complex $L$ and  we apply Theorem~\ref{t:asphericalcorner}, then $f^*(X)$ is a space constructed previously in \cite{cdstrict}:  the strict hyperbolization of the cubical complex $f^*(\square)$. 
\end{examples}

\subsection{Closed manifolds}\label{ss:examples}
\begin{proposition}\label{p:closed1}
Suppose $L$ is a simplicial complex with vertex set $I$ and $\bman=\{(M(i),\partial M(i))\}_{i\in I}$ is a collection of manfolds with boundary.  If $L$ is a triangulation of a sphere, then $\zl\bman$ is a closed manifold.
\end{proposition}

\begin{proposition}\label{p:closed2}
Suppose $\{(M_i,\partial M_i)\}_{i\in I(n)}$ is a corner of manifolds structure on $M$.  
\begin{enumeratea}
\item
Suppose $W$ is a finite Coxeter group of rank $n+1$.  Then $\cu(W,M)$ is a closed manifold.
\item
Suppose $L$ is a triangulation of a sphere and $f:L\to \gD$ is a coloring (so that the pullback can be defined).  Then $f^*(M)$ is a closed manifold.
\end{enumeratea}
\end{proposition}

As we pointed out in Remark~\ref{r:310}, part (a) of Proposition~\ref{p:closed2} is a special case of part (b).  Proposition~\ref{p:closed1}  is also essentially a special case of part (b).  So, we only prove part (b).

\begin{proof}[Proof of Proposition~\ref{p:closed2}\,(b)]
It follows from \eqref{e:pullback} in the definition of a pullback that $f^*(M)$ is a union of pieces of the form, $J\times \mfj$, where $J\in \cs(L)$ and where $\mfj$ denotes the relative interior of $M_{\fjcheck}$.  A neighborhood of $\emptyset \times M_{I(n)}$ in $f^*(L)$ has the form $\cone (L)\times M_{I(n)}$ which is a manifold since $\cone(L)$ is a disk.  Similarly, $J\times \mfj$ has a neighborhood of the form $\cone(\Lk (J,L)) \times \mfj$ and this is a manifold since $\cone(\Lk(J,L)$ is a homology manifold and becomes locally Euclidean once we take its product with a Euclidean space of dimension at least $|J|$.  (Note: $\dim \mfj=|J| +\dim M_{I(n)} \ge |J|$.)
\end{proof}

Theorem~\ref{t:asphericalflag} and Proposition~\ref{p:closed1} can be  combined as follows.

\begin{proposition}\label{p:polyman}
Suppose  a flag complex $L$ is a triangulation of a sphere, that $\bman=\{(M(i),\partial M(i))\}_{i\in I}$ is a family of manifolds with boundary indexed by the vertex set $I$ of $L$ and, as in Definition~\ref{d:apair}, that $(M(i),\partial M(i))$ is an aspherical pair.  Then the polyhedral product $\zl\bman$ is a closed aspherical manifold.
\end{proposition}

From now on suppose  $M$ is a corner of manifolds with set of codimension one faces, $\cm=\{M_i\}_{i\in I(n)}$, satisfying the following conditions (cf.\ \S\ref{ss:aspflag}).
\begin{enumeratei'}
\item\label{i:1}
$M$ is aspherical.
\item
The induced mirror structure $\wcm$ on the universal cover is acyclic.
\item
$\wcm$ is  flag.
\end{enumeratei'}

Theorem~\ref{t:flag} and  Proposition~\ref{p:closed2}\,(a) can be  combined as follows.
\begin{proposition}\label{p:refltrick}
With hypotheses as above, suppose $(W,S)$ is a spherical Coxeter system of rank $n+1$.  Then $\cu(W,M)$ is a closed aspherical manifold.
\end{proposition}

Finally, we can combine Theorem~\ref{t:asphericalcorner} and Proposition~\ref{p:closed2}\,(b) to get the following proposition.

\begin{proposition}\label{p:pullbackman}
With hypotheses as above, suppose 
$L$ is a flag triangulation of a sphere and that $f:L\to \gD$ is a coloring.  Then $f^*(M)$ is a closed aspherical manifold.
\end{proposition}

\end{document}


\title{Criteria for asphericity of polyhedral products: 
corrigenda to ``Right-angularity, flag complexes, asphericity''}
\author{Michael W. Davis 
\and {Peter H. Kropholler} }
\date{}
\date{\today} 
\maketitle
\begin{abstract}
Given a simplicial complex $L$ with vertex set $I$ and a family 
$\ab=\{(A(i),B(i))\}_{i\in I}$ of pairs of spaces with base points $*_i\in B(i)$, there is a definition of the ``polyhedral product'' $\zx$ of $\ab$ with respect to $L$. Sometimes this is called a ``generalized moment angle complex''.
This note concerns two refinements to earlier work of the first author.  First, when $L$ is infinite, the definition of polyhedral product needs clarification.  Second, the earlier paper omitted some subtle parts of the necessary and sufficient conditions for polyhedral products to be aspherical. Correct versions of these necessary and sufficient conditions are given in the present paper.
	 \smallskip

	\noindent
	\textbf{AMS classification numbers.} Primary: 20F65, 57M07, 57M10. \\
	Secondary: 20F36, 20F55, 20J05. 
	\smallskip
	
	\noindent
	\textbf{Keywords:} aspherical space, graph product, polyhedral product, right-angled Coxeter group, right-angled building. 
\end{abstract}

\paragraph{The definition of a polyhedral product}
Let $L$ be a simplicial complex. Suppose that  $\ab=\{(A(i),B(i))\}_{i\in I}$ is a collection of pairs of spaces, with base points  $*_i\in B(i)$, indexed by the vertex set $I$ of $L$.  As a set, the \emph{polyhedral product}  of $\ab$ with respect to $L$ is the subset $\zx$ of the product $\prod_{i\in I} A(i)$ consisting of the points
$(x_i)_{i\in I}$ satisfing the following two conditions:
	\begin{enumerateI} 
	\item\label{i:finite}
	$x_i= *_i$ for all but finitely many $i$,
	\item\label{i:simplex}
	$\{i\in I\mid x_i\notin B(i)\}$ is a simplex of $L$.
	\end{enumerateI}
For pairs $M<M'$ of subcomplexes of $L$, the base points determine natural inclusions $\bA^M\hookrightarrow\bA^{M'}$.
When $L$ is finite, the polyhedral product inherits its topology from the product $\prod_{i\in I} A(i)$. In general $\zx$ acquires the topology of the colimit $\displaystyle\lim_{\longrightarrow}\bA^M$ as $M$ runs through the finite subcomplexes. 

The above definition is given in \cite[\S 4.1]{km}.  When $L$ is finite, condition~\eqref{i:finite} is vacuous and the definition reduces to the one in \cite[\S 2.1]{rafca}. The restriction to finite $L$ in \cite{rafca} is implicit but inessential. 

In \cite{bbcg2} and \cite{rafca}, the notation $\cz_L(\mathbf{A},\mathbf{B})$ was used for the polyhedral product, here denoted simply by $\zx$.

It is assumed that all the spaces $A(i)$ and $B(i)$ are CW-complexes. In particular, simply connected aspherical spaces are contractible.

\paragraph{Asphericity of polyhedral products}  We assume  each $A(i)$ is path connected.  Let $G_i$ denote the fundamental group $\pi_1(A(i), *_i)$ and let $E_i=\pi_1(A(i),B(i), *_i)$,  so that $E_i$ is a $G_i$-set. Define a subset  $I''$ of $I$, by 
	\begin{equation}\label{e:I''}
	I''=\{i\in I\mid \Card (E_i)\neq1\}
	\end{equation}
and let $L''$ be the full subcomplex of $L$ with vertex set $I''$. 

In the statement and proof of \cite[Theorem~2.22]{rafca}  two implicit assumptions are made:
	\begin{enumeratea}
	\item\label{i:inject}
	for each path component of $B(i)$, the induced map of fundamental groups $\pi_1(B(i))\to \pi_1(A(i))$ is injective, and
	\item\label{i:proper}
	when $B(i)$ is path connected, the image of $\pi_1(B(i))$ is a proper subgroup of $\pi_1(A(i))$.
	\end{enumeratea}

Under assumptions \eqref{i:inject} and \eqref{i:proper}, it is shown in \cite{rafca} that $\zx$ is aspherical if and only if the following three conditions hold:
	\begin{itemize}
	\item
	Each $A(i)$ is aspherical.
	\item
	If a vertex $i$ is not joined to every other vertex by an edge, then $(A(i),B(i))$ is an aspherical pair. (See Definition~\ref{d:pair}, below, for the meaning of ``aspherical pair.'')
	\item
	$L$ is a flag complex.
	\end{itemize}
In proving the necessity of these conditions,
assumption \eqref{i:inject} leads to a conclusion that certain pairs $(A(i),B(i))$ must be aspherical; without \eqref{i:inject},  the $(A(i),B(i))$ are only ``nearly aspherical'' (cf.~Definition~\ref{d:pair}, below).	
Assumption \eqref{i:proper} implies that $I=I''$ and hence $L=L''$.  When \eqref{i:proper} does not hold further modifications are needed in \cite[Theorem~2.22]{rafca} (cf.~Example~\ref{ex:KL} below).
 
\begin{example}\label{ex:KL} (cf. \cite[Example~2.1]{rafca}).
If for each $i\in I$,  $(A(i),B(i))=([0,1],\{1\})$, then the polyhedral product $([0,1],\{1\})^L$ is elsewhere called a \emph{chamber} and denoted $K(L)$. Since $K(L)$ is isomorphic to the cone on the barycentric subdivision of $L$, the space $([0,1],\{1\})^L$ is contractible for any $L$.  
\end{example}

\begin{example} \label{ex:cont}
More generally, if, for each $i$,  $A(i)$ is contractible and $B(i)$ is acyclic,  then $\zx$ is simply connected and the the natural map 
 $\zx \to ([0,1],\{1\})^L$ induces an isomorphism on homology; hence, $\zx$ is contractible for any $L$.
\end{example}

\begin{example}\label{ex:simplex}
If $L$ is a simplex, then $\zx = \prod_{i\in I} A(i)$.  Hence, $\zx$ is aspherical if and only if each $A(i)$ is aspherical -- no condition is needed on the $B(i)$.
\end{example}

\begin{definition} (cf.~\cite[p.~249]{rafca}).  A vertex of a simplicial complex $L$ is \emph{conelike} if it is joined by an edge to every other vertex of $L$ 
\end{definition}

\begin{definition}\label{d:pair}(cf.~\cite[Definition~2.21]{rafca}). Suppose $(A,B)$ is a pair of spaces with $A$ path connected.  Let $\wa$ denote the universal cover of $A$ and let $\wb$ be the inverse image of $B$ in $\wa$.  Then $(A,B)$ is an \emph{aspherical pair} if $\wa$ is contractible and each component of $\wt{B}$ is contractible (in other words, $A$ is aspherical, each path component of $B$ is aspherical and the fundamental group of each such component maps injectively into $\pi_1(A)$). The pair $(A,B)$ is  \emph{nearly aspherical} if $\wa$ is contractible and if each path component of $\wb$ is acyclic.
\end{definition}

For any subset $J\leq I$, write $\langle J\rangle$ for the full subcomplex of $L$ spanned by $J$.  Given a field $\ff$ define disjoint subsets of $I-I''$:
	\begin{align*}
	I_1&:=\{i\in I-I''\mid \wb(i) \text{ is acyclic with coefficients in $\zz$}\}\\
	I_2&:=\{i\in I-I''\mid \wb(i)\text{ is not acyclic with coefficients in $\ff$}\},
	\end{align*}
Put
	\begin{equation}\label{e:prime}
	I'=I_2 \cup I'' \quad \text{and}\quad L'=\langle I'\rangle.
	\end{equation}
From now on we assume the following:  
\medskip

$(*)$\ \ \emph{for some field $\ff$, $I_1 \cup I_2$ is a partition of $I-I''$}. 
\medskip

\noindent
Although this assumption seems fairly mild, without it there is a problem finding a clean statement of our theorem (cf.~Example~\ref{ex:notacyclic} below).

The corrected statement of Theorem~2.22 in \cite{rafca} is the following.

\begin{theorem}\label{thm}
Suppose that $I$ is the vertex set of a simpliciial complex $L$, that  $\ab=\{(A(i), B(i))\}_{i\in I}$ is a family of pairs and that condition $(*)$ holds.  The following four conditions are necessary and sufficient for $\zx$ to be aspherical.
	\begin{enumeratei}
	\item \label{i:A}
	Each $A(i)$ is aspherical.
	\item \label{i:notcone}
	If $i\in I'$ is not conelike in $L'$, then $(A(i), B(i))$ is  nearly aspherical.
	\item \label{i:I2}
	The vertices of $I_2$ are conelike in $L'$.  Moreover, 
	$\langle I_2 \rangle$ is a simplex and $L'$ is the join $\langle I_2 \rangle*L''$ 	
	\item \label{i:flag}
	$L'$ is a flag complex.
	\end{enumeratei}    
\end{theorem}

\begin{Remark} When $I= I''$, Theorem~2.22 of \cite{rafca} is almost correct as stated, except that in condition \eqref{i:notcone} the phrase ``aspherical pair'' should be replaced by ``nearly aspherical pair.''
Condition (iii) of Theorem 2.22  reads $L$ instead of $L'$ (as in condition~\eqref{i:flag} above).  The point is that for $\zx$ to be aspherical it is not necessary that $L$  be a flag complex but only that $L'$ be a flag complex.  This should have been obvious from Example~2.1 in \cite{rafca} (which is essentially Example~\ref{ex:KL}, above). 
\end{Remark}

An important observation in the proof of Theorem~2.22 of \cite{rafca} is that if $J< I$, then the natural map $r:\zx \to \bA^{\langle{J}\rangle}$ induced by the projection, $\prod_{i\in I} A(i) \to \prod_{i\in J} A(i)$ is a retraction (cf.~\cite[Lemma~2.5]{rafca}).  Applying this to the case where $J$ is a singleton, we see that if $\zx$ is aspherical, then so is each $A(i)$.

As in \cite{rafca}, Theorem~\ref{thm} is proved by reducing to the special case where each pair $(A(i),B(i))$ has the form $(\cone(E_i), E_i)$, for some discrete set $E_i$, where $\cone(E_i)$ means the cone on $E_i$.  Put $\mathbf C :=\{(\cone(E_i),E_i)\}_{i\in I}$.  Consider the polyhedral product $\mathbf C^L$ and its universal cover $\widetilde{\mathbf C^L}$.  A key point is that $\widetilde{\mathbf C^L}$ is the standard realization of a right-angled building (abbreviated as $\rab$) if and only if the following two conditions hold:
	\begin{itemize}
	\item
	for each $i\in I$, $\Card(E_i)\ge 2$, and
	\item
	$L$ is a flag complex.
	\end{itemize}
This is stated in \cite[Example~2.8]{rafca} and a proof is given in \cite{d09}.  Let $W$ denote the right-angled Coxeter group associated to the $1$-skeleton of $L$ and let $\cac$ denote the underlying chamber system of the $\rab$.

One flaw in \cite{rafca} traces back to the assertion in the proof of Lemma~2.11, when $L$ is a flag complex, the chamber system $\cac$ 
(obtained, as in \cite{d09}, via a covering of a nonstandard realization of $\prod_{i\in I} E_i$) is a $\rab$.  
However, this is only true when each $E_i$  has more than one element, i.e., when the statement after the first bullet point holds. 

Let $I'' <I$ be the subset defined by \eqref{e:I''}.  As before, put $L''=\langle I''\rangle$. The correct version of Lemma 2.11 in \cite{rafca} is the following.

\begin{lemma}\label{l:econe}
Let $\mathbf C :=\{(\cone(E_i),E_i)\}_{i\in I}$. The polyhedral product $\mathbf C^L$ is aspherical if and only if $L''$ is a flag complex.
\end{lemma}

\begin{proof}
First observe that the 
inclusion $\mathbf C^{L''}\hookrightarrow \mathbf C^L$ is a homotopy equivalence.  The reason is that there is a retraction $r\colon\mathbf C^L\to \mathbf C^{L''}$ induced by the projection $\prod_{i\in I} \cone (E_i) \to \prod_{i\in I''} \cone (E_i)$; moreover, $r$ is a deformation retraction.  (Since, for each $i\in I-I''$,  $E_i$ is a singleton, $\cone(E_i)$ deformation retracts onto $E_i$.) 

The family $\{E_i\}_{i\in I''}$ and the right-angled Coxeter group $W''$ associated to the $1$-skeleton of $L''$ define a right-angled $\rab$ with underlying chamber system $\cac''$ (defined as in \cite{d09} by taking the universal cover of a nonstandard realization of the product building, $\prod_{i\in I''} E_i$).
Suppose $L''$ is a flag complex.  As was pointed out above, $\widetilde{\mathbf C^{L''}}$ is the standard realization of $\cac''$. Therefore, it is $\cat(0)$ and hence contractible (cf.~\cite[Thm.{}18.3.1, p.{}338]{dbook}).  So, $\mathbf C^{L''}$ is aspherical and by the previous paragraph, $\mathbf C^L$  also is aspherical.

Conversely,  suppose $L''$ is not a flag complex.  Then $\widetilde{\mathbf C^{L''}}$ is a nonstandard realization of $\cac''$, i.e., $\widetilde{\mathbf C^{L''}}= \cu(\cac'',K(L''))$, where $K(L'')$ is the chamber dual to $L''$ and where $\cu(\cac'',K(L''))$ means the ``basic construction'' defined as in \cite{dbook}.  An apartment in this nonstandard realization is $\cu(W'',K(L''))$.   Since $L''$ is not a flag complex, the realization of the apartment, $\cu(W'',K(L''))$, is not contractible (see \cite[Thm.{}9.1.4, p.{}167]{dbook}).  Since $\cu(\cac'',K(L''))$ retracts onto the (nonstandard) realization of any apartment, we see that $\widetilde{\mathbf C^{L''}}$ cannot be contractible and hence,  $\widetilde{\mathbf C^{L}}$ also cannot be contractible.
\end{proof}

\begin{lemma}\label{l:I2}
Suppose $L$ is a simplicial complex consisting of two vertices.  Let $I=\{1,2\}$ be its vertex set and let $\ab=\{(A(i),B(i))\} _{i\in I}$ be two pairs of spaces such that $A(i)$ path connected, 
neither $\wb(1)$ nor $\wb(2)$ is acyclic (i.e., $I_1=\emptyset$), and condition (*) holds.
Then $\zx$ is aspherical if and only if, for $i=1,2$, $(A(i), B(i))$ is nearly aspherical.  
\end{lemma}

\begin{proof}
For $i=1,2$, let $E_i$ denote the set of path components $\wb(i)$ and let $\mathbf C$ denote the family $\{(\cone(E_i),E_i)\}_{i\in I}$.  Then $\mathbf C^L$ is a complete bipartite graph $\gO$ with vertex set $E_1\coprod E_2$.  Similarly, ${\widetilde{\bA}}^L$ is a graph of spaces for the same graph $\gO$:  the set of vertex spaces for one type is  $\{\wb(1)_e\times \wa(2)\}_{e\in E_1}$ and for the other type it is $\{\wa(1)\times \wb(2)_f\}_{f\in E_2}$.  Here $\wb(1)_e$ (resp. $\wb(2)_f$) means the path component corresponding to $e\in E_1$ (resp. $f\in E_2$).  Since $\wa(i)$ is simply connected, $\pi_1(\widetilde {\bA}^L)\cong \pi_1(\gO)$. (The reason is that the cover of $\widetilde {\bA}^L$ corresponding to the universal cover of $\gO$ is a tree of spaces with simply connected vertex spaces.)

Suppose $\zx$ is aspherical.  Since $A(i)$ is a retract of 
${\bA}^L$, it also is aspherical; so its universal cover $\wa(i)$ is contractible.  Since the universal cover of 
${{\widetilde{\bA}}^L}$ is contractible, the natural map ${\widetilde{\bA}}^L \to \gO$ is a homotopy equivalence.  
In particular, it induces an isomorphism on homology.  On the other hand, when each $\wa(i)$ is contractible, there is a formula for the homology of any polyhedral product (cf.~\cite[Thm.{}2.21]{bbcg2}).  In the case at hand, it gives that 
${\widetilde{\bA}}^L$ has the same homology as the suspension of a smash product, $S^0 *(\wb(1)\wedge \wb(2))$.  So, $\wb(1)\wedge \wb(2)$ must have the same homology as the discrete space $E_1\wedge E_2$.   By condition (*) there is a field $\ff$ so that $I_2\cup I''$ is a partition of $\{1,2\}$.  With coefficients in $\ff$, the homology of a smash product is concentrated in degree $0$ if and only if both factors are concentrated in degree $0$ or if one of them is acyclic over $\ff$.
If $I_2=\{1,2\}$, then both $ \wb(1)$ and $\wb(2)$ are connected and have nontrivial homology in positive degrees with $\ff$ coefficients, which is impossible.  If $I_2$ is a singleton, then the other element of $\{1,2\}$ is in $I''$ (since $I_1=\emptyset$) and we again get a contradiction.  So, $I_2=\emptyset$, i.e., neither $\wb(1)$ nor $\wb(2)$ is connected.  In the case at hand, this means that, for $i=1, 2$, each component of $\wb(i)$ is acyclic (over $\zz$).  Hence $(\wa(i), \wb(i))$ is nearly aspherical.

Conversely, suppose  each $(\wa(i),\wb(i))$ is nearly aspherical. 
There is a $\pi_1(\gO)$-equivariant map of trees of spaces, from the universal cover of ${{\widetilde{\bA}}^L}$ to $\widetilde{\mathbf C^L}$, which induces an isomorphism on homology. Hence the universal cover of
${{\widetilde{\bA}}^L}$ is contractible.
\end{proof}

\begin{example}\label{ex:notacyclic}
For $i=1,2$, suppose $\wa(i)$ is contractible and that the reduced homology of $\wb(i)$ is $p_i$-torsion for distinct primes $p_1$ and $p_2$. Then  
$\wb(1)\wedge \wb(2)$ is acyclic over $\zz$.  Hence, for $L=\langle\{1,2\}\rangle$ as in Lemma~\ref{l:I2}, ${\widetilde\bA}^L$ is contractible.
\end{example}

\begin{proof}[Proof of Theorem~\ref{thm}]
First suppose $\zx$ is aspherical.  As noted previously, this implies \eqref{i:A}. 
Suppose $i\in I'$ is not conelike in $L'$.  Then there is another vertex $j\in I'$ not connected by an edge to $i$. Since $\bA^{\langle{\{i, j\}}\rangle}$ is a retract of $\zx$, it is also aspherical. 
Applying Lemma~\ref{l:I2}, we see that \eqref{i:notcone} holds.  For each $i\in I_2$, $\wb(i)$ is path connected and not acyclic, so,  $(A(i),B(i))$ is not a nearly aspherical pair; hence, $i$ must be conelike.  By assumption, there is a field $\ff$ so that for each $i\in I_2\cup I''$, $\wb(i)$ is not acyclic with coefficients in $\ff$. (If $i\in I''$, then $\wb(i)$ is not connected.) Suppose there is an ``empty simplex'' in $\langle I_2\cup I''\rangle$, i.e., a subset $J< I'$ so that $\langle J\rangle$ is isomorphic to the boundary of a $k$-simplex, $k\ge 2$. Since the reduced homology (with coefficients in $\ff$) of smash product of  $\{\wb(i)\}_{i\in J}$ is nonzero, it follows from the formula in \cite[Thm.~2.21]{bbcg2} that this homology, shifted up in degree by $k$, appears as a direct summand of $H_*({\widetilde\bA}^{\langle{I_2 \cup I''}\rangle};\ff)$.
This shows, first of all, that $J$ cannot be a subset of $I_2$.  (Since $\bA^{\langle I_2\rangle}$ is aspherical, $\widetilde{\bA}^{\langle{I_2}\rangle}$ must be  contractible.)  Hence, $\langle I_2\rangle$ is a simplex.  Furthermore, $J$ cannot decompose as $J_2 \cup J''$, where $J_2=J\cap I_2$ and $J''= J\cap I''$ are both nonempty.  (This is because $\bA^{\langle{J}\rangle}$ is aspherical and has the same fundamental group as $\mathbf C^{\langle J''\rangle}$; hence, the homology of $\widetilde\bA^{\langle J\rangle}$ must vanish in degrees $\ge \Card (J'')$.) Therefore, $L'$ is the join $\langle I_2\rangle *L''$, i.e., \eqref{i:I2} holds.  A similar argument shows that $\langle I''\rangle$ cannot contain an empty $k$-simplex for $k\ge 2$ with vertex set $J''$, for then the universal cover of $\widetilde\bA^{\langle J'\rangle}$ has a nonzero homology in degree $k$ (cf. the second paragraph of the proof of Theorem~2.22 in \cite[p.~249]{rafca}).  Therefore, $L''$ and $L'$ are both flag complexes, i.e., \eqref{i:flag} holds.

Conversely, suppose conditions \eqref{i:A}, \eqref{i:notcone}, \eqref{i:I2}, \eqref{i:flag} hold.  Let $J$ be the set of conelike vertices in $L'$.  Since $L'$ is flag, 
$\langle J\rangle$ is a simplex and $L'= \langle J\rangle *\langle I'-J \rangle$.  Since each $\wa(i)$ is contractible, so is $\widetilde\bA^{\langle J\rangle} = \prod_{i\in J} \wa(i)$.  Therefore, $\widetilde\bA^{\langle L'\rangle} = \widetilde\bA^{\langle J\rangle} \times \widetilde\bA^{\langle I'-J\rangle}$ is homotopy equivalent to 
$\widetilde\bA^{\langle I'-J\rangle}$.  Since the universal cover of ${\mathbf C}^{\langle I'-J\rangle}$ is contractible 
(being the standard realization of a $\rab$) and since for each $i\in I'-J$, the natural map $(\wa(i), \wb(i)) \to (\cone (E_i), E_i)$ induces an isomorphism on homology, it follows that $\widetilde\bA^{\langle I'-J\rangle}$ is contractible and therefore, so is $\widetilde\bA^{\langle L'\rangle}$.
For any set of vertices $H<I$, we have that the reduced homology of the smash product of $\{\wb(i)\}_{i\in H}$ is trivial if $H\cap I_1 \neq \emptyset$.   
It follows from the formula in \cite[Thm.~2.21]{bbcg2} that $\widetilde\bA^L \to \widetilde\bA^{L'}$ induces an isomorphism on homology and hence is a homotopy equivalence.  So, $\widetilde\bA^L$ is aspherical and therefore, so is $\bA^L$.
\end{proof}

\paragraph{Acknowledgement}We thank Boris Okun for his help with the proof of Lemma~2.

Michael W. Davis, Department of Mathematics, The Ohio State University, 231 W. 18th Ave., Columbus Ohio 43210, \url{davis.12@math.osu.edu}

Peter H. Kropholler, Mathematical Sciences, University of Southampton, Southampton SO17 1BJ \url{p.h.kropholler@soton.ac.uk}